\documentclass[10pt]{amsart}
\usepackage{latexsym}
\usepackage{amsfonts}

\usepackage{amsmath,amsthm,amssymb}

\author[Ilya Kapovich]{Ilya Kapovich}

\address{\tt Department of Mathematics, University of Illinois at
 Urbana-Champaign, 1409 West Green Street, Urbana, IL 61801, USA
 \newline http://www.math.uiuc.edu/\~{}kapovich/} \email{\tt
  kapovich@math.uiuc.edu}

\author[Tatiana Nagnibeda]{Tatiana Nagnibeda}

\address{\tt
Section de math\'ematiques,
Universit\'e de Gen\`eve,
2-4, rue du Li\`evre, c.p. 64,
1211 Gen\`eve, Switzerland
\newline http://www.unige.ch/math/folks/nagnibeda}
\email{\tt tatiana.smirnova-nagnibeda@unige.ch}

\title[Subset currents on free groups]{Subset currents on free groups}

\newtheorem{thm}{Theorem}[section] \newtheorem{lem}[thm]{Lemma}
\newtheorem{cor}[thm]{Corollary} 
\newtheorem{prop}[thm]{Proposition} \theoremstyle{definition}
\newtheorem{defn}[thm]{Definition}
\newtheorem{notation}[thm]{Notation}
\newtheorem{conv}[thm]{Convention} \newtheorem{rem}[thm]{Remark}

\newtheorem{propdfn}[thm]{Proposition-Definition}

 \newtheorem{prob}[thm]{Problem}

\def\strutdepth{\dp\strutbox}
\def \ss{\strut\vadjust{\kern-\strutdepth \sss}}
\def \sss{\vtop to \strutdepth{
\baselineskip\strutdepth\vss\llap{$\diamondsuit\;\;$}\null}}

\def\strutdepth{\dp\strutbox}
\def \sst{\strut\vadjust{\kern-\strutdepth \ssss}}
\def \ssss{\vtop to \strutdepth{
\baselineskip\strutdepth\vss\llap{$\spadesuit\;\;$}\null}}

\def\strutdepth{\dp\strutbox}
\def \ssh{\strut\vadjust{\kern-\strutdepth \sssh}}
\def \sssh{\vtop to \strutdepth{
\baselineskip\strutdepth\vss\llap{$\heartsuit\;\;$}\null}}

\newcommand{\R}{\mathbb R}
\newcommand{\Z}{\mathbb Z}

\def\epsilon{\varepsilon}
\def\phi{\varphi}
\def\Pr{\mathbb P}

\newcommand{\supp}{\mbox{Supp}}
\newcommand{\Curr}{\mbox{Curr}}
\newcommand{\Out}{\mbox{Out}}
\newcommand{\Aut}{\mbox{Aut}}
\newcommand{\Inn}{\mbox{Inn}}

\newcommand{\vol}{\mbox{vol}}

\newcommand{\bd}{\partial}
\newcommand{\bdsq}{\partial^2}

\newcommand{\PCN}{\Pr\Curr(\FN)}

\newcommand{\FN}{F_N}   
\newcommand{\cvn}{\mbox{cv}_N}
\newcommand{\cvnbar}{\overline{\mbox{cv}}_N}
\newcommand{\CVN}{\mbox{CV}_N}
\newcommand{\CVNbar}{\overline{\mbox{CV}}_N}

\newcommand{\gcn}{\mathcal S\Curr(\FN)}
\newcommand{\pgcn}{\mathbb P\gcn}

\newcommand{\rk}{\rm rk}
\newcommand{\rrk}{\overline{\rm rk}}

\newcommand\isom{\mathrel{\text{%
   \setbox0\hbox{$\rightarrow$}%
   \rlap{\hbox to \wd0{\hss\raisebox{0.9\height}{$\sim$}\hss}}\box0
}}}

\begin{document}

\begin{abstract}
We introduce and study the space $\gcn$ of \emph{subset currents} on the free group $F_N$, and, more generally, on a word-hyperbolic group. A subset current on $F_N$ is a positive $F_N$-invariant locally finite Borel measure on the space $\mathfrak C_N$ of all closed subsets of $\partial F_N$ consisting of at least two points. The well-studied space $\Curr(F_N)$ of geodesics currents, positive $F_N$-invariant locally finite Borel measures defined on pairs of different boundary points, is contained in the space of subset currents as a closed $\Bbb R$-linear $\Out(F_N)$-invariant subspace. Much of the theory of $\Curr(F_N)$ naturally extends to the $\gcn$ context, but new dynamical, geometric and algebraic features also arise there. While geodesic currents generalize conjugacy classes of nontrivial group elements, a subset current is a measure-theoretic generalization of the conjugacy class of a nontrivial finitely generated subgroup in $F_N$.  If a free basis $A$ is fixed in $F_N$, subset currents may be viewed as $F_N$-invariant measures on a \lq\lq branching\rq\rq\ analog of the geodesic flow space for $F_N$, whose elements are infinite subtrees (rather than just geodesic lines) of the Cayley graph of $F_N$ with respect to $A$. Similarly to the case of geodesics currents, there is a continuous $\Out(F_N)$-invariant  \lq\lq co-volume form\rq\rq\ between the Outer space $\cvn$ and the space $\gcn$ of subset currents. Given a tree $T\in \cvn$ and the \lq\lq counting current\rq\rq\ $\eta_H\in \gcn$ corresponding to a finitely generated nontrivial subgroup $H\le F_N$, the value $\langle T, \eta_H\rangle$ of this intersection form turns out to be equal to the co-volume of $H$, that is the volume of the metric graph $T_H/H$, where $T_H\subseteq T$ is the unique minimal $H$-invariant subtree of $T$. However, unlike in the case of geodesic currents, the co-volume form $\cvn\times\gcn\to [0,\infty)$ does not extend to a continuous map $\cvnbar\times\gcn\to [0,\infty)$. 
\end{abstract}

\thanks{The first author was supported by the NSF
  grant DMS-0904200. Both authors acknowledge the support of the Swiss National
  Foundation for Scientific Research}


\maketitle

\tableofcontents

\section{Introduction}\label{sec:intro}

Geodesic currents were introduced in the context of hyperbolic surfaces by Bonahon in the papers \cite{Bo86,Bo88}  where they were used to study the Teichm\"uller space and mapping class groups, with various applications to 3-manifolds. 
In the context of free groups, geodesic currents were first investigated by Reiner Martin in his 1995 PhD thesis~\cite{Martin}, but have only become the object of systematic study in the last five years, leading to a number of interesting recent applications and developments.

A geodesic current may be thought of as a measure-theoretic analog of the notion of a conjugacy class in the group, or of a free homotopy class of a closed curve in the surface.  More formally, a \emph{geodesic current} on a free group $F_N$ is a positive Borel measure on $\bdsq F_N:=\{(x,y)\in \bd F_N\times \bd F_N : x\ne y\}$ which is locally finite (i.e., finite on compact subsets), $F_N$-invariant and invariant with respect to the \lq\lq flip map\rq\rq\ $\bdsq F_N\to \bdsq F_N$, $(x,y)\mapsto (y,x)$. Equivalently, it is a positive Borel locally finite $F_N$-invariant measure on the space of 2-element subsets of $\partial F_N$.  

This paper is devoted to the study of a natural generalization of the space of geodesic currents obtained by replacing the space of 2-element subsets of $\partial F_N$ by the space $\mathfrak C_N$ of all closed subsets $S\subseteq \partial F_N$ such that $S$ consists of at least two points. The space $\mathfrak C_N$ has a natural topology given by the Hausdorff distance for the subsets of $\partial F_N$, where $\partial F_N$ is endowed with the standard visual metric provided by any choice of basis in $F_N$ (this topology is independent of the choice of the basis and coincides with the Vietoris topology). Similarly to the Cantor set $\bd F_N$, the space $\mathfrak C_N$ is locally compact and totally disconnected. See Subsection~\ref{subsec:subsets} for details about the space $\mathfrak C_N$.
A \emph{subset current} on $F_N$ is a positive locally finite $F_N$-invariant Borel measure on $\mathfrak C_N$, and we consider the space $\gcn$ of all subset currents on $F_N$ equipped with the weak-$*$ topology and with the $\R_{\ge 0}$-linear structure. 

This definition naturally extends to the case where $F_N$ is replaced by an arbitrary word-hyperbolic group, for details see Problem \ref{prob:surface} below. 

The space $\gcn$ of subset currents admits a natural $\Out(F_N)$-action by $\R_{\ge 0}$-linear homeomorphisms, and the space of geodesic currents $\Curr(F_N)$ is canonically embedded in $\gcn$ as a closed $\Out(F_N)$-invariant $\R_{\ge 0}$-linear subspace. We show below that subset currents are  measure-theoretic analogs of conjugacy classes of  nontrivial finitely generated subgroups in $F_N$, with geodesic currents corresponding to cyclic subgroups, and we study this connection in detail. 

The space $\Curr(F_N)$ of all geodesic currents on $F_N$ is a natural counterpart for the Culler-Vogtmann Outer space $\cvn$, and the present paper explores deeper levels of this interaction in a more general context of subset currents.  The Outer space $\cvn$, introduced in~\cite{CV},  is a free group \lq\lq cousin\rq\rq\ of the Teuchm\"uller space and consists of $F_N$-equivariant isometry classes of free minimal discrete isometric actions of $F_N$ on $\R$-trees. Points of $\cvn$ may also be thought of as \emph{marked metric graph structures} on $F_N$, see Section~\ref{sec:background} below for details. The space $\cvn$ comes equipped with a natural $\Out(F_N)$-action that factors through to the action on the projectivized Outer space $\CVN$ (whose points are homothety classes of elements of $\cvn$). 
The closure $\cvnbar$ of $\cvn$ with respect to the equivariant Gromov-Hausdorff convergence topology is an important object in the study of dynamics of $\Out(F_N)$ and is known to consist of $F_N$-equivariant isometry classes of all \emph{very small} minimal isometric actions of $\Out(F_N)$ on $\R$-trees. The projectivization $\CVNbar$ of $\cvnbar$ is a compact finite-dimensional space analog to Thurston's compactification of the Teichm\"uller space. 
Compared to the Teichm\"uller space, the geometry of the Outer space remains much less understood.


In this regard studying the interaction between the Outer space and the space of currents proved to be quite useful. This interaction is primarily given by the \emph{geometric intersection form} or \emph{length} pairing: in~\cite{KL2} Kapovich and Lustig proved that there exists a unique continuous map
\[
\langle \, , \rangle: \cvnbar\times \Curr(F_N)\to \R_{\ge 0}
\] 
which is $\Out(F_N)$-invariant, $\R_{>0}$-homogeneous with respect to the first argument, $\R_{\ge 0}$-linear with respect to the second argument, and has the property that for every nontrivial element $g\in F_N$ and every $T\in\cvnbar$ one has
\[
\langle T, \eta_g\rangle =||g||_T.
\]
Here $||g||_T=\inf_{x\in T} d_T(x,gx)$ is the \emph{translation length} of $g$ with respect to $T$, and $\eta_g\in \Curr(F_N)$ is the \emph{counting current} associated to $g$ (see Subsection \ref{subsec:geodcurr} for the definition).
The set of scalar multiples of all counting currents is a dense subset of $\Curr(F_N)$, which in particular justifies thinking about the notion of a current as generalizing that of a conjugacy class. This approach provided a number of useful recent applications to the study of the dynamics and geometry of $\Out(F_N)$, such as the results of Kapovich and Lustig about various analogs of the curve complex in the free group case~\cite{KL2}; a construction by Bestvina and Feighn, for a given finite collection of iwip elements in $\Out(F_N)$, of a Gromov-hyperbolic graph with an isometric $\Out(F_N)$-action, where these iwip automorphisms act as hyperbolic isometries; a result of Hamenst\"adt~\cite{Ha} about \lq\lq lines of minima\rq\rq\ in Outer space; the work of Clay and Pettet~\cite{CP} on realizability of an arbitrary matrix from $GL(N,\Z)$ as the abelianization action of a hyperbolic iwip element of $\Out(F_N)$,  and others (see, for example, \cite{Ka3,Ka4,KL1,KL2,KL3,KL4,KL5,CHL3,Fra,KN,KN2,CK}).

As noted above, in this paper we extend this framework in a way that allows us to study the dynamics of the action of $\Out(F_N)$ on conjugacy classes of finitely generated subgroups of $\Out(F_N)$ that are not necessarily cyclic. 

In the context of subset currents, we similarly define, given a finitely generated nontrivial subgroup $H\le F_N$, the \emph{counting current} $\eta_H\in\gcn$. For the case where $H=\langle g\rangle\le F_N$ is infinite cyclic, we actually get $\eta_H=\eta_g$. For a finitely generated nontrivial $H\le F_N$,  $\eta_H\in \gcn$ is defined (see Definition~\ref{defn:count}) as a sum of atomic measures on the limit set of $H$ and its conjugates. It is then shown (see Theorem~\ref{thm:cc}), that it can equivalently be understood in terms of the corresponding Stallings core graphs.

If a free basis $A$ is fixed in $F_N$, a convenient basis of topology is formed by cylinder sets. A subset cylinder $\mathcal SCyl_A(K)\subseteq \mathfrak C_N$ (see Definition \ref{defn:gcyl} below) is determined by a finite non-degenerate subtree $K$ of the Cayley tree $X_A$ of $F_N$ with respect to $A$, and it consists of all those $S\in \mathfrak C_N$ such that the convex hull of $S$  in $X_A$ contains $K$ and such that every bi-infinite geodesic in $X_A$ with both endpoints in $S$ which intersects $K$ in a non-degenerate segment, enters and exits $K$ through vertices of degree 1 in $K$.

We develop the appropriate notion of an \emph{occurrence}  of a finite subtree $K\subseteq X_A$ in a Stallings core graph $\Delta$ (see Definition~\ref{defn:occur}). Namely, an occurrence of $K$ in $\Delta$ is a locally injective label-preserving morphism from $K$ into $\Delta$ which is a local homeomorphism at every point of $K$ except for the terminal vertices of leaves of $K$. Apart from being a useful tool in dealing with counting currents, the language of occurrences in core graphs produces an interesting model of non-linear (that is, not based on a segment of $\Z$) words, with a good notion of a subword (or \lq\lq factor\rq\rq ) in such a word. Studying such models is an active subject of research in combinatorics of words (see, for example, \cite{ABFJ}). 

As noted earlier, in the context of $\Curr(F_N)$, a basic fact of the theory states that the set of all \emph{rational currents}, that is of  all $\R_{\ge 0}$-scalar multiples of the counting currents $\eta_g$, where $g\in F_N, g\ne 1$, is a dense subset of $\Curr(F_N)$.  Proofs of this fact are usually relying, in an essential way, on the \lq\lq commutative\rq\rq\  nature of the dynamical systems associated with $\Curr(F_N)$. Thus, given a free basis $A$ of $F_N$, one can naturally view a geodesic current on $F_N$ as a positive shift-invariant Borel finite measure  on the space of bi-infinite freely reduced words over $A^{\pm 1}$. Under this correspondence, the counting currents correspond exactly to the shift-invariant measures supported by periodic orbits of the shift map, and the density of the set of rational currents is then a consequence of classical results about density of periodic orbits.  
In the context of $\gcn$ the situation is considerably more complicated, since the symbolic dynamical  systems corresponding to subset currents are no longer \lq\lq commutative\rq\rq\ in nature. Geometrically, a 2-element subset of $\bd F_N$ determines an infinite unparameterized geodesic line in the Cayley graph $X_A$ of $F_N$ with respect to a free basis $A$, so that a geodesic current may be thought of as an $F_N$-invariant measure on the space of unparameterized geodesic lines in $X_A$.  Similarly, an element $S\in \mathfrak C_N$ determines an infinite subtree  $Y\subseteq X_A$ without degree-one vertices, namely, the convex hull of $S$ in $X_A$.  Thus, once $A$ is chosen, a subset current translates into a positive finite Borel measure on the space $\mathcal T_1(X_A)$ of infinite subtrees of $X_A$ containing $1\in F_N$ and without degree-one vertices. $F_N$-invariance of the current implies that the corresponding measure on  $\mathcal T_1(X_A)$ is invariant with respect to the root change. The space $\mathcal T_1(X_A)$ may be viewed as a \lq\lq branching\rq\rq\ analog of the geodesic flow space for $X_A$, since its elements are infinite trees rather than lines.

Recent results of Bowen~\cite{Bow03,Bow09} and Elek~\cite{Elek} about approximability of such measures on the space of rooted trees are applicable to our set-up and allow us to conclude in Theorem~\ref{thm:dense} that the set $\{r\eta_H| r\ge 0, H\le F_N \text{ is nontrivial and finitely generated}\}$ of \emph{rational subset currents} is dense in $\gcn$.

The notion of a subset current is related to the study of \lq\lq invariant random subgroups\rq\rq .  If $G$ is a locally compact group, an \emph{invariant random subgroup} is a probability measure on the space $S(G)$ of all closed subgroups of $G$, such that this measure is invariant with respect to the conjugation action of $G$ on $S(G)$. 

The study of invariant random subgroups in various contexts goes back to the work of Stuck and Zimmer~\cite{SZ94} and has recently become an active area of research, see for example~\cite{AGV12,Bow10,Bow12,DDMN,D02,Gr11,Sa11,Vershik,Ve11}.  If $G=F_N$ and $A$ is a free basis of $F_N$, one can view an invariant random subgroup on $F_N$ as a measure on the space of all rooted Stallings core graphs labelled by $A$, invariant with respect to root change. Note that the Stallings core graph corresponding to a nontrivial subgroup is never a tree, and, moreover, every edge in this graph is contained in some nontrivial immersed circuit.  By contrast, as noted above and as we explain in greater detail in Subsection~\ref{subsec:root} below, a subset current on $F_N$ can be viewed as a measure on the space $\mathcal T_1(X_A)$ of infinite trees which is invariant with respect to root change. We plan to investigate deeper connections between these two notions in a future work.

In this paper we also construct (see Section~\ref{sec:intform}) a continuous $\Out(F_N)$-invariant \emph{co-volume form} 
\[
\langle\, ,\, \rangle: \cvn\times\gcn\to \R_{\ge 0}.
\]
It has properties similar to that of the geometric intersection form on geodesic currents, except that instead of translation length it computes the \emph{co-volume}: for any $T\in\cvn$ and a nontrivial finitely generated subgroup $H\le F_N$, we have $\langle T, \eta_H\rangle= \vol(H\setminus T_H)$ where $T_H$ is the convex hull of the limit set of $H$, and is also the unique minimal $H$-invariant subtree of $T$. For an infinite cyclic $H=\langle g\rangle\le F_N$ we have $||g||_T=\vol(H\setminus T_H)$ since in this case the quotient graph $H\setminus T_H$ is  a circle. By contrast to the result of Kapovich and Lustig~\cite{KL2} for ordinary geodesic currents, we prove (Theorem~\ref{thm:disc} below) that for $N\ge 3$ the co-volume form $\cvn\times\gcn\to \R_{\ge 0}$ does not extend to a continuous map $\cvnbar\times\gcn\to \R_{\ge 0}$.

For a finitely generated subgroup $H\le F_N$, the \emph{reduced rank}, $\rrk(H)$, is defined as $\max\{\rk(H)-1,0\}$, where $\rk(H)$ is the cardinality of a free basis of $H$.
It turns out that reduced rank uniquely extends to a continuous $\Out(F_N)$-invariant $\R_{\ge 0}$-linear functional $\rrk: \gcn\to \R_{\ge 0}$, such that for every nontrivial finitely generated $H\le F_N$ we have $\rrk(\eta_H)=\rrk(H)$. As we note in Section~\ref{sec:problems}, there are likely deeper connections arising here with the study of intersections of finitely generated subgroups of free groups.


More open problems regarding subset currents are formulated in  Section~\ref{sec:problems}. We believe that subset currents exhibit considerably more interesting and varied geometric and dynamical behavior than geodesic currents, and can provide new and interesting information about $\Out(F_N)$. 
Indeed, we hope this paper to serve as a starting point for investigating the space of subset currents in its different aspects, such as  
$\Out(F_N)$-related questions, connections with invariant random subgroups, connections to the study of ergodic properties of subgroup actions on $\partial F_N$ and their Schreier graphs (see~\cite{GKN}), etc.

We are very grateful to Lewis Bowen for illuminating and  helpful conversations and to the referee for useful comments.
The first author also thanks Patrick Reynolds for pointing out that for a fixed finitely generated subgroup $H\le F_N$, the co-volume function $\cvnbar\to \R_{\ge 0}, T\mapsto ||H||_T$, is not continuous on $\cvnbar$. 

\section{Outer space and the space of geodesic currents}\label{sec:background}

We give here only a brief overview of basic facts related to Outer space and the space of geodesic currents. We refer the reader to~\cite{CV,Ka2} for more detailed background information.

\subsection{Conventions regarding graphs}\label{subsec:conv}

A \emph{graph} is a 1-complex. The set of $0$-cells of a graph
$\Delta$ is denoted $V\Delta$ and its elements are called
\emph{vertices} of $\Delta$. The closed 1-cells of a graph $\Delta$ are
called \emph{topological edges} of $\Delta$. The set of all
topological edges is denoted $E_{top}\Delta$.  We will sometimes call the
open 1-cells of $\Delta$ \emph{open edges} of $\Delta$. The interior of every topological edge is
homeomorphic to the interval $(0,1)\subseteq \mathbb R$ and thus
admits exactly 2 orientations (when considered as a 1-manifold). We
call a topological edge endowed with the choice of an orientation on
its interior an \emph{oriented edge} of $\Delta$. The set of all oriented edges
of $\Delta$ is denoted $E\Delta$. For an oriented edge $e\in E\Delta$
changing its orientation to the opposite produces another oriented
edge of $\Delta$ denoted $e^{-1}$ and called the \emph{inverse} of
$e$. Thus ${}^{-1}:E\Delta\to E\Delta$
is a fixed-point-free involution. 
For every oriented edge $e$ of $\Delta$ there are naturally defined
(and not necessarily distinct)
vertices $o(e)\in V\Delta$, called the \emph{origin} of $e$, and
$t(e)\in V\Delta$, called the \emph{terminus} of $e$, satisfying
$o(e^{-1})=t(e)$, $t(e^{-1})=o(e)$. An \emph{orientation} on a graph $\Delta$ is a partition $E\Delta=E^+\Delta\sqcup E^-\Delta$, where for every $e\in E\Delta$ one
of the edges $e, e^{-1}$ belongs to $E^+\Delta$ and the other edge belongs to  $E^-\Delta$.

An \emph{edge-path} $\gamma$ of \emph{simplicial length} $|\gamma|=n\ge 1$
in $\Delta$ is a sequence of oriented edges 
\[
\gamma=e_1,\dots, e_n
\]
such that $t(e_i)=o(e_{i+1})$ for $i=1,\dots, n-1$. We say that
$o(\gamma):=o(e_1)$ is the \emph{origin} of $\gamma$ and that
$t(\gamma)=t(e_n)$ is the \emph{terminus} of $\gamma$. 
An edge-path is
called \emph{reduced} if it does not contain a back-tracking, that is
a sub-path of the form $ee^{-1}$, where $e\in E\Delta$.

To every edge-path $\gamma=e_1,\dots, e_n$ in $\Delta$ there is a naturally
associated continuous map $\widehat\gamma:[0,n]\to \Delta$ with
$\widehat\gamma(0)=o(\gamma)$ and $\widehat\gamma(n)=t(\gamma)$.

A \emph{graph morphism} $f : \Delta \to \Delta'$ is a continuous map from a
graph $\Delta$ to a graph $\Delta'$ that maps vertices of $\Delta$ to
vertices of $\Delta'$ and such that each open edge of $\Delta$ is
mapped homeomorphically to an open edge of $\Delta'$.  Thus $f$
induces natural maps $f: E\Delta\to E\Delta'$ and $f:V\Delta\to
V\Delta'$ such that $f(e^{-1})=\left(f(e))\right)^{-1}$,
$o(f(e))=f(o(e))$ and $t(f(e))=f(t(e))$ for every $e\in E\Delta$.

\subsection{Markings}\label{subsec:mark}

Let $N\ge 2$. We fix a free basis $A=\{a_1,\dots, a_N\}$ of $F_N$. We
also let $R_N$ to denote the wedge of $N$ loop-edges at a vertex $x_0$. We
identify $F_N$ with $\pi_1(R_N, x_0)$ by mapping each $a_i\in B$ to
one of the loop edges in $R_N$. We fix this identification
$F_N=\pi_1(R_N,x_0)$ for the remainder of the paper. The graph $R_N$
will be referred to as the \emph{standard $N$-rose}.

A \emph{marking} on $F_N$ is an
isomorphism $\alpha:F_N\isom\pi_1(\Gamma)$, where $\Gamma$ is a finite
connected graph without degree-1 vertices.

To each marking $\alpha$ we also associate a continuous map
$\widehat\alpha: R_N\to \Gamma$ such that $\widehat \alpha$ is a
homotopy equivalence and such that $\widehat\alpha_\#=\alpha$.
Two markings
$\alpha_1:F_N\isom\pi_1(\Gamma_1)$ and $\alpha_2:F_N\isom\pi_1(\Gamma_2)$
are \emph{equivalent} if there exists a graph isomorphism
$j:\Gamma_1\to\Gamma_2$ such that for the associated continuous maps
 $\widehat\alpha_{1}:R_N\to\Gamma_1$ and $\widehat\alpha_{2}:R_N\to \Gamma_2$,
the maps $j\circ \widehat\alpha_{1}$ and
$\widehat\alpha_{2}$ are freely homotopic. 
We will usually only be interested in the equivalence class of a
marking, and for that reason an explicit mention of a base-point in the graph
$\Gamma$ in the definition of a marking will almost always be omitted.

If $\alpha:F_N\isom\pi_1(\Gamma)$ is a marking, then $\alpha$ defines an
$F_N$-equivariant quasi-isometry between $F_N$ and $\widetilde\Gamma$
(endowed with the simplicial metric, giving every edge of
$\widetilde\Gamma$ length $1$). This quasi-isometry induces an
$F_N$-equivariant homeomorphism $\partial F_N\to\partial
\widetilde\Gamma$. When talking about markings, we will usually implicitly assume that $\partial
F_N$ is identified with $\partial\widetilde\Gamma$ via this homeomorphism
and write $\partial F_N=\partial\widetilde\Gamma$.

Each marking $\alpha:F_N\isom\pi_1(\Gamma)$ provides a Hausdorff metric $d_{\alpha}$ on $\partial F_N$ as follows.  For $\xi,\zeta\in \partial
F_N$, put  $d_\alpha(\xi,\zeta)=\frac{1}{2^M}$ where $M$ is the length
of the maximal common initial segment of the geodesic rays $[x_0,\xi)$ and $[x_0,\zeta)$ in $\widetilde \Gamma$.

A \emph{metric graph structure} on a graph $\Delta$ is a function
$\mathcal L:E\Delta\to\ (0,\infty)$ such that $\mathcal L(e)=\mathcal
L(e^{-1})$ for every $e\in E\Delta$. 
Equivalently, we may think of a metric graph structure on $\Delta$ as
a function $\mathcal L: E_{top}\Delta\to (0,\infty)$. 
For an edge-path
$\gamma=e_1,\dots, e_n$ in $\Delta$ the \emph{$\mathcal L$-length} of
$\gamma$ is $\mathcal L(\gamma):=\sum_{i=1}^n \mathcal L(e)$.

A \emph{marked metric graph structure} on $F_N$ is a pair $(\alpha,
\mathcal L)$ where  $\alpha:F_N\isom\pi_1(\Gamma)$ is a marking on
$F_N$ and $\mathcal L$ is a metric graph structure on $\Gamma$.
Two marked metric graph structures $(\alpha_1:F_N\isom\pi_1(\Gamma_1),
\mathcal L_1)$ and $(\alpha_2:F_N\isom\pi_1(\Gamma_2), \mathcal L_2)$
are \emph{equivalent} if there exists a graph isomorphism
$j:\Gamma_1\to\Gamma_2$ such that $j: (\Gamma_1, \mathcal L_1)\to
(\Gamma_2, \mathcal L_2)$ is an isometry and such that for the associated continuous maps
 $\widehat\alpha_{1}:R_N\to\Gamma_1$ and $\widehat\alpha_{2}:R_N\to \Gamma_2$,
the maps $j\circ \widehat\alpha_{1}$ and
$\widehat\alpha_{2}$ are freely homotopic. 
Note that if  $(\alpha_1:F_N\isom\pi_1(\Gamma_1),
\mathcal L_1)$ and $(\alpha_2:F_N\isom\pi_1(\Gamma_2), \mathcal L_2)$
are equivalent marked metric graph structures on $F_N$ then the
markings $\alpha_1,\alpha_2$ are equivalent.

\subsection{Outer space}\label{subsec:outer}


Let $N\ge 2$. The \emph{Outer space} $\cvn$ consists of all minimal free and discrete isometric
actions on $F_N$ on $\mathbb R$-trees (where two such actions are
considered equal if there exists an $F_N$-equivariant isometry between
the corresponding trees).
There are several different topologies on
$\cvn$ that are known to coincide, in particular the equivariant
Gromov-Hausdorff convergence topology and the so-called \emph{length function} topology. 

Every $T\in \cvn$ is uniquely
determined by its \emph{translation length function}
$||.||_T:F_N\to\mathbb R$, where $||g||_T$ is the translation length
of $g$ on $T$. Two trees $T_1,T_2\in \cvn$ are close if the functions
$||.||_{T_1}$ and $||.||_{T_1}$ are close point-wise on a large ball in
$F_N$. The closure $\cvnbar$ of $\cvn$ in either of these two
topologies is well-understood and known to consist precisely of all
the so-called \emph{very small} minimal isometric actions of $F_N$ on
$\mathbb R$-trees, see \cite{BF93} and \cite{CL}.

The automorphism group $\Aut(F_N)$ has a natural continuous \emph{right} action
on $\cvnbar$ (that leaves $\cvn$ invariant) defined as follows: for $T\in \cvn$ and $\phi\in \Aut(F_N)$
we have $||g||_{T\phi}=||\phi(g)||_T$, where $g\in F_N$. In terms of
tree actions, $T\phi$ is equal to $T$ as a metric space, but the
action of $F_N$ is modified as: $g\underset{T\phi}{\cdot}
x=\phi(g)\underset{T}{\cdot} x$ where $x\in T$, $g\in F_N$. It is not
hard to see that the subgroup $\Inn(F_N)\le \Aut(F_N)$ of  inner
automorphisms is contained in the kernel of the action of $\Aut(F_N)$
on $\cvnbar$. Hence this action quotients through to the action of
$\Out(F_N)$ on $\cvnbar$, where $\cvn\subseteq\cvnbar$ is an
$\Out(F_N)$-invariant dense subset.  The right action of $\Out(F_N)$ on $\cvnbar$ can be converted into a left action as follows: for $T\in \cvnbar$ and $\phi\in \Out(F_N)$ put $\phi T:=T\phi^{-1}$.

The
\emph{projectivized Outer space} $\CVN=\mathbb P\cvn$ is defined
as the quotient $\cvn/\sim$ where for $T_1\sim T_2$ whenever
$T_2=cT_1$ in $\cvn$ for some $c>0$. One similarly defines the projectivization
$\CVNbar=\mathbb P\cvnbar$ of $\cvnbar$ as
$\cvnbar/\sim$ where $\sim$ is the same as above. The space
$\CVNbar$ is compact and contains $\CVN$ as a dense
$\Out(F_N)$-invariant subset. The compactification $\CVNbar$ of $\CVN$
is a free group analog of the Thurston compactification of the
Teichm\"uller space. For $T\in \cvnbar$ its $\sim$-equivalence class
is denoted by $[T]$, so that $[T]$ is the image of $T$ in $\CVNbar$.

Every marked metric graph structure $(\alpha:F_N\isom \pi_1(\Gamma),
\mathcal L)$ defines a point in $\cvn$ as follows. Consider the universal covering tree
$X=\widetilde\Gamma$ and let $d_\mathcal L$ be the metric on $X$
obtained by giving every edge of $\widetilde\Gamma$ the same length as the $\mathcal L$-length of
its projection in $\Gamma$. Then $X$ is an $\mathbb R$-tree, and the
action of $F_N$ on $X$ via $\alpha$ by covering
transformations is a free minimal discrete isometric action on
$(X,d_\mathcal L)$. Thus $(X,d_\mathcal L)$, equipped with this action of $F_N$, is a point of $\cvn$.

It is well-known that for two marked metric graph structures $(\alpha_1:F_N\isom \pi_1(\Gamma_1),
\mathcal L_1)$ and $(\alpha_2:F_N\isom \pi_1(\Gamma_2),
\mathcal L_2)$ we have $(\widetilde\Gamma_1,d_{\mathcal
  L_1})_{\alpha_1}=(\widetilde\Gamma_2,d_{\mathcal L_2})_{\alpha_2}$
in $\cvn$ if and only if $(\alpha_1,\mathcal L_1)$ is equivalent to
$(\alpha_2,\mathcal L_2)$.
Moreover, every point of $\cvn$ comes from some marked metric graph
structure on $F_N$. Namely, if $T\in \cvn$, take $\Gamma=F_N\setminus
T$ and endow the edges of $\Gamma$ with the same lengths as their
lifts in $T$.  Since the action of $F_N$ on $T$ is free and discrete,
there is a natural identification $F_N$ with $\pi_1(\Gamma)$, giving
us a marking $\alpha:F_N\isom \pi_1(\Gamma)$ on $F_N$.  This yields a
marked metric graph structure $(\alpha:F_N\isom \pi_1(\Gamma),
\mathcal L)$ such that $T= (\widetilde\Gamma,d_{\mathcal
  L})_{\alpha}$ in $\cvn$.



\subsection{Geodesic currents}\label{subsec:geodcurr}

Let $\partial^2 F_N:=\{ (x,y)| x,y\in \partial F_N, x\ne y\}$. The action of $F_N$ by translations on its hyperbolic boundary $\partial F_N$ defines a natural diagonal action of $F_N$ on $\partial^2 F_N$. A \emph{geodesic current} on $F_N$ is a positive Borel measure on $\partial^2 F_N$, which is locally finite (that is finite on all compact subsets),  $F_N$-invariant and is also invariant under the \lq\lq flip\rq\rq\ map $\partial^2 F_N\to \partial^2 F_N$, $(x,y)\mapsto (y,x)$. 


The space $\Curr(F_N)$ of all geodesic currents on $F_N$ has a natural $\mathbb R_{\ge 0}$-linear structure and is equipped with the weak-* topology of point-wise convergence on continuous functions. Any choice of a marking on $F_N$ allows one to think about geodesic currents as systems of nonnegative weights satisfying certain Kirchhoff-type equations; see \cite{Ka2} for details. We briefly recall this construction for the case where $X_A\in \cvn$ is the Cayley tree corresponding to a free basis $A$ of $F_N$. For a non-degenerate geodesic segment $\gamma=[p,q]$ in $X_A$ the \emph{two-sided cylinder} $Cyl_A(\gamma)\subseteq \partial^2 F_N$ consists of all $(x,y)\in \partial^2 F_N$ such that the geodesic from $x$ to $y$ in $X_A$ passes through $\gamma=[p,q]$. Given a nontrivial freely reduced  word $v\in F(A)=F_N$ and a current $\mu\in \Curr(F_N)$, the \lq\lq weight\rq\rq\  $(v;\mu)_A$ is defined as $\mu(Cyl_A(\gamma))$ where $\gamma$ is any segment in the Cayley graph $X_A$ labelled by $v$ (the fact that the measure $\mu$ is $F_N$-invariant implies that a particular choice of $\gamma$ does not matter). A current $\mu$ is uniquely determined by a family of weights $\big((v;\mu)_A\big)_{v\in F_N-\{1\}}$. The weak-* topology on $\Curr(F_N)$ corresponds to point-wise convergence of the weights for every $v\in F_N, v\ne 1$.

There is a natural left action of $\Out(F_N)$ on $\Curr(F_N)$ by continuous linear transformations. Specifically, let $\mu\in Curr(F_N)$, $\phi\in \Out(F_N)$ and let $\Phi\in Aut(F_N)$ be a representative of $\phi$ in $Aut(F_N)$. Since $\Phi$ is a quasi-isometry of $F_N$, it extends to a homeomorphism of $\partial F_N$ and, diagonally, defines a homeomorphism of $\partial^2 F_N$. The measure $\phi\mu$ on $\partial^2 F_N$ is defined as follows. For a Borel subset $S\subseteq \partial^2 F_N$ we have $(\phi\mu)(S):=\mu(\Phi^{-1}(S))$. One then checks that $\phi\mu$ is a current and that it does not depend on the choice of a representative $\Phi$ of $\phi$. 

The \emph{space of projectivized geodesic currents} is defined as $\mathbb P\Curr(F_N)=\Curr(F_N)-\{0\}/\sim$ where $\mu_1\sim\mu_2$ whenever there exists $c>0$ such that $\mu_2=c\mu_1$. The $\sim$-equivalence class of $\mu\in \Curr(F_N)-\{0\}$ is denoted by $[\mu]$.  The action of $\Out(F_N)$ on $\Curr(F_N)$ descends to a continuous action of $\Out(F_N)$ on $\mathbb P\Curr(F_N)$. The space $\mathbb PCurr(F_N)$ is compact.

For every $g\in F_N, g\ne 1$ there is an associated \emph{counting current} $\eta_g\in Curr(F_N)$. If $A$ is a free basis of $F_N$ and the conjugacy class $[g]$ of $g$ is realized by a \lq\lq cyclic word\rq\rq\ $W$ (that is a cyclically reduced word in $F(A)$ written on a circle with no specified base-vertex), then for every nontrivial freely reduced word $v\in F(A)=F_N$ the weight $( v;\eta_g)_A$ is equal to the total number of occurrences of $v^{\pm 1}$ in $W$ (where an occurrence of $v$ in $W$ is a vertex on $W$ such that we can read $v$ in $W$ clockwise without going off the circle). We refer the reader to \cite{Ka2} for a detailed exposition on the topic. By construction, the counting current $\eta_g$ depends only on the conjugacy class $[g]$ of $g$ and it also satisfies $\eta_g=\eta_{g^{-1}}$. One can check~\cite{Ka2} that for $\phi\in \Out(F_N)$ and $g\in F_N, g\ne 1$ we have $\phi\eta_g=\eta_{\phi(g)}$. Scalar multiples $c\eta_g\in \Curr(F_N)$, where $c\ge 0$, $g\in F_N, g\ne 1$ are called \emph{rational currents}. A key fact about $\Curr(F_N)$ states that the set of all rational currents is dense in $\Curr(F_N)$. The set $\{[\eta_g\: g\in F_N, g\ne 1\}$ is dense in $\mathbb P\Curr(F_N)$.

\subsection{Intersection form}\label{subsec:intersection}

In \cite{KL2} Kapovich and Lustig  constructed a natural geometric \emph{intersection form} that pairs trees and currents:

\begin{prop}\label{prop:int}\cite{KL2}
Let $N\ge 2$. There exists a unique continuous map $\langle , \rangle : \cvnbar \times \Curr(F_N)\to \mathbb R_{\ge 0}$ with the following properties:
\begin{enumerate}
\item We have $\langle T, c_1\mu_1+c_2\mu_2\rangle=c_1\langle T,\mu_1\rangle+c_2\langle T,\mu_2\rangle$ for any $T\in \cvnbar$, $\mu_1,\mu_2\in \Curr(F_N)$, $c_1,c_2\ge 0$.
\item We have $\langle cT, \mu\rangle=c\langle T,\mu\rangle$ for any
  $T\in \cvnbar$, $\mu\in \Curr(F_N)$ and $c\ge 0$.
\item We have $\langle T\phi,\mu\rangle=\langle T, \phi\mu\rangle$ for
  any $T\in \cvnbar$, $\mu\in \Curr(F_N)$ and $\phi\in \Out(F_N)$.
\item We have $\langle T, \eta_g\rangle=||g||_T$ for any $T\in \cvnbar$ and $g\in F_N, g\ne 1$.
\end{enumerate}
\end{prop}

\section{The space of subset currents}\label{sec:subsetcurr}

\subsection{The space $\mathfrak C_N$}\label{subsec:subsets}

Recall that for a Hausdorff topological space $Y$,  the so-called \emph{hyper space} $\mathcal H(Y)$ consists of all non-empty closed subsets of $Y$. The space $\mathcal H(Y)$ comes equipped with the \emph{Vietoris topology}, which has the basis consisting of all sets of the form

$$\langle U_1,...,U_n\rangle = \{B\in \mathcal H(Y) | B\subset U_1\cup ... \cup U_n ; B\cap U_i \neq \emptyset , i=1,...,n\} ,$$

where $\{U_1,...,U_n\}$ is a family of open subsets in $Y$.

If $Y$ is a compact metrizable space, then the Vietoris topology  coincides with the Hausdorff topology given by the Hausdorff distance between closed subsets of $Y$, and in this case $\mathcal H(Y)$ is also compact (see e.g. \cite{Encycl}, Chapter b-6). If $Y$ is totally disconnected then $\mathcal H(Y)$ is also totally disconnected.  

\begin{defn}[Space of closed subsets of the boundary]\label{defn:SN}
Let $N\ge 3$. We denote by $\mathfrak C_N$ the set of all closed subsets $S\subseteq \partial F_N$ such that $\#S\ge 2$.
Thus $\mathfrak C_N\subseteq \mathcal H(\bd F_N)$ and we endow $\mathfrak C_N$ with the subspace topology inherited from the Vietoris topology on $\mathcal H(\bd F_N)$. 
\end{defn}

In view of the above remarks,  the topology on $\mathfrak C_N$ coincides with the Hausdorff topology given by the Hausdorff distance $D_\alpha$ on $\mathfrak C_N$ with respect to the metric $d_\alpha$ on $\partial F_N$ corresponding to any marking $\alpha$ of $F_N$ (see \ref{subsec:mark}). 
It is also not difficult to check directly that the topology on $\mathfrak C_N$ defined by the metric $D_\alpha$ does not depend on
the choice of the marking.

The definition also straightforwardly implies that $\mathfrak C_N$ is locally compact and totally disconnected.

The condition that $\#S\ge 2$ is an important non-degeneracy condition for setting up the notion of a subset current, analogous to the assumption that $\xi\ne \zeta$ for $(\xi,\zeta)\in \partial^2 F_N$ in the definition of a geodesic current.

The space $\mathfrak C_N$ that interests us in this paper, is the complement in $\mathcal H (\partial F_N)$ of the closed subspace consisting of $1$-element subsets $\{y\}$, $y\in\partial F_N.$ 


It will be useful in our study of the space $\gcn$ to understand the topology on $\mathfrak C_N$  more explicitly. For this reason we describe the construction of \lq\lq cylindrical\rq\rq\ subsets in $\mathfrak C_N$. 

\begin{defn}[Cylinders]\label{defn:gcyl} 
Let $\alpha:F_N\to \pi_1(\Gamma)$ be a marking on $F_N$, and let $X=\widetilde \Gamma$ be the universal cover of $\Gamma$. 

We give each edge of $\Gamma$ and of $X$ length 1, so that $X$ can also be considered an element of $\cvn$.

Let $K\subset X$ be a non-degenerate finite simplicial subtree of $X$, that is a finite simplicial subtree with at least two distinct vertices of degree 1. Let $e_1,\dots, e_n$ be all the terminal edges of $K$, that is oriented edges whose terminal vertices are precisely all the vertices of $K$ of degree 1. (Note that $n\ge 2$ by the assumption on $K$.) For each of the edges $e_i$ denote by $Cyl_X(e_i)\subseteq \partial F_N$ the homeomorphic image of the subset in $\partial X$ consisting of all equivalence classes of geodesic rays in $X$ beginning with the edge $e_i$.

Define the {\it subset cylinder} $\mathcal SCyl_\alpha(K)$ to be the set $\langle Cyl_X(e_1),...,Cyl_X(e_n)\rangle \subset \mathfrak C_N$. 
Thus for a closed subset $S\subseteq \bd F_N$ with $\#S\ge 2$ we have $S\in \mathcal SCyl_\alpha(K)$ if and only if the following hold:
\begin{enumerate}
\item The subset $S\subseteq \partial F_N$ is closed.
\item We have \[S\subseteq \cup_{i=1}^n Cyl_X(e_i).\]
\item For each $i=1,\dots, n$ we have \[S\cap Cyl_X(e_i)\ne \emptyset.\]
\end{enumerate}
\end{defn}

The following key basic fact is a straightforward exercise in unpacking the definitions:

\begin{prop}\label{prop:topSN}
Let $\alpha$, $\Gamma$ and $X$ be as in Definition~\ref{defn:gcyl}.
Then
\begin{enumerate}
\item For every non-degenerate finite subtree $K\subseteq X$ the subset $\mathcal SCyl_\alpha(K)\subset \gcn$ is compact and open.
\item The collection of all $\mathcal SCyl_\alpha(K)$, where $K$ varies over all non-degenerate finite subtrees of $X$, forms a basis for the topology on $\mathfrak C_N$ given in Definition~\ref{defn:SN}.
\end{enumerate}
\end{prop}

\begin{notation}\label{notation:edges}
Let $X$ be a simplicial tree and let $e$ be an oriented edge of $X$. Denote by $q(e)$ the set of all oriented edges $e'$ in $X$ such that $e,e'$ is a reduced edge-path in $X$. 

For any set $B$ we denote by $P_+(B)$ the set of all nonempty subsets of $B$.
\end{notation}

The following simple lemma is key for the Kirchhoff-type formulas for subset currents (see Proposition~\ref{prop:kirch} below).

\begin{lem}\label{lem:disj} Let $\Gamma, X$, $K\subset X$ be as in Definition~\ref{defn:gcyl} and let $e_1,\dots, e_n$ be the terminal edges of $K$, as in Definition \ref{defn:gcyl}.
Then for every $i=1,\dots, n$ we have
\[
\mathcal SCyl_\alpha(K)=\sqcup_{U\in P_+(q(e_i))} \mathcal SCyl_\alpha(K\cup U) .
\] 
\end{lem}

\begin{proof}
Fix $i, 1\le i\le n$.
It is obvious from the definitions that for every such nonempty $U$   we have $\mathcal SCyl_\alpha(K\cup U)\subseteq \mathcal SCyl_\alpha(K)$ and hence the union of $\mathcal SCyl_\alpha(K\cup U)$ over all such $U$ is contained in $\mathcal SCyl_\alpha(K)$. 

Let $S\in \mathcal SCyl_\alpha(K)$ be arbitrary. By the definition of $\mathcal SCyl_\alpha(K)$ we have $S\cap Cyl_X(e_i)\ne\emptyset$. Let $U$ be the set of all edges $e\in EX$ such that there exists a point $\xi\in S$ that contains the geodesic ray in $X$ beginning with the edge-path $(e_i,e)$. 
Then $S\subseteq \mathcal SCyl_\alpha(K\cup U)$.

It follows that 
\[
\mathcal SCyl_\alpha(K)\subseteq \cup_U \mathcal SCyl_\alpha(K\cup U)
\]
with $U$ as required. 
We leave it as an exercise to the reader verifying that the union on the right-hand side in the above formula is a disjoint union.    
\end{proof}

\subsection{Subset currents}\label{subsec:currents}

Observe that the left translation action of $F_N$ on $\partial F_N$ naturally
extends to a left translation action by homeomorphisms on $\mathfrak
C_N$. We can now define the main notion of this paper:

\begin{defn}[Subset currents] 
A \emph{subset current} on $F_N$ is a positive Borel measure $\mu$ on $\mathfrak C_N$ which is $F_N$-invariant and locally finite (i.e., finite on all compact subsets of $\mathfrak C_N$).

The set of all subset currents on $F_N$ is denoted $\gcn$. The space $\gcn$ is endowed with the natural weak-* topology of convergence of integrals of continuous functions with compact support.

Define an equivalence relation $\sim$ on $\gcn-\{0\}$ as: $\mu\sim\mu'$ if $\mu=c\mu'$ for some $c>0$, where $\mu,\mu'\in \gcn-\{0\}$.  For a nonzero $\mu\in\gcn$ the $\sim$-equivalence class of $\mu$ is denoted be $[\mu]$ and is called the \emph{projective class} of $\mu$. Put $\pgcn=\left(\gcn-\{0\}\right)/\sim$ and endow $\pgcn$ with the quotient topology. 
\end{defn}

It is not hard to check (c.f. the proof by Francaviglia~\cite{Fra} of a similar statement for ordinary geodesic currents on $F_N$) that the weak-* topology on $\gcn$ can be described in more concrete terms:
\begin{prop}\label{top:GC}
Let $\alpha:F_N\to \pi_1(\Gamma)$ be a marking on $F_N$, and let $X=\widetilde \Gamma$ be the universal cover of $\Gamma$.
\begin{enumerate}
\item Let $\mu, \mu_n\in \gcn$, where $n=1,2,\dots$. Then $\lim_{n\to\infty} \mu_n=\mu$ in $\gcn$ if and only if for every finite non-degenerate subtree $K$ of $X$ we have
\[
\lim_{n\to\infty} \mu_n(\mathcal SCyl_\alpha(K))= \mu(\mathcal SCyl_\alpha(K)).
\]

\item For each  finite non-degenerate subtree $K$ of $X$ the function 
\[
\gcn\to\mathbb R, \quad \mu\mapsto \mu(\mathcal SCyl_\alpha(K))
\]
is continuous on $\gcn$.
\item Let $\mu\in \gcn$. For $\epsilon>0$ and an integer $M\ge 1$ let $U(M,\epsilon,\mu)$ be the set of all $\mu'\in \gcn$ such that for every finite non-degenerate subtree $K$ of $X$ with at most $M$ edges we have
\[
\left| \mu'(\mathcal SCyl_\alpha(K))- \mu(\mathcal SCyl_\alpha(K)) \right|<\epsilon.
\] 
Then the family $\{U(M,\epsilon, \mu): M\ge 1, 0<\epsilon<1\}$ forms a basis of open neighborhoods for $\mu$ in $\gcn$.
\end{enumerate}
\end{prop}

The following key observation follows directly from the definition of subset cylinders (see Definition \ref{defn:gcyl} above) and from the fact that subset currents are $F_N$-invariant.

\begin{propdfn}[Weights]\label{propdfn:weights}
Let $\alpha: F_N\rightarrow \pi_1(\Gamma)$ be a marking on $F_N$, let $X=\widetilde \Gamma$, and let $K$ be a finite non-degenerate subtree of $X$. Then for every element $g\in F_N$ we have
\[
g\mathcal SCyl_\alpha(K)=\mathcal SCyl_\alpha(gK)
\]
and
\[
\mu(\mathcal SCyl_\alpha(K))=\mu(g\mathcal SCyl_\alpha(K)).\tag{$\ast$}
\]
We denote \[ (K;\mu)_\alpha:=\mu(\mathcal SCyl_\alpha(K))\] and call it the \emph{weight} of $K$ in $\mu$.
For a given finite subtree $K$ of $X$, we denote the $F_N$-translation class of $K$ by $[K]$ (so that $[K]$ consists of all the translates of $K$ by elements of $F_N$). 
We put
\[
( [K]; \mu)_\alpha:=( K; \mu)_\alpha
\] 
and call it the \emph{weight} of $[K]$ in $\mu$. In view of $(\ast)$, the weight $( [K]; \mu)_\alpha$ is well-defined and does not depend on the choice of $K$ in $[K]$. It defines a continuous function on $\gcn$.
\end{propdfn}

\begin{cor}
Let $\alpha, \Gamma$ and $X$ be as in Proposition~\ref{prop:topSN}. Let $K$ be a finite non-degenerate subtree of $X$.  Then the function $f:\gcn\to\R$ given by $f(\mu)=(K; \mu)_\alpha$, where $\mu\in \gcn$, is continuous.
\end{cor}

\begin{notation}\label{not:e}
Let $\alpha, \Gamma$ and $X$ be as in Proposition~\ref{prop:topSN}.  For every topological edge $e\in E_{top}\Gamma$ denote by $\widetilde e$ the subgraph of $X=\widetilde \Gamma$ consisting of any lift of $e$.  By $F_N$-invariance of subset currents, the weight $(\widetilde e; \mu)_\alpha$ depends only on $\mu, \alpha$ and $e$ and does not depend on the choice of a lift $\widetilde e$ of $e$ to $X$. We thus denote $(e; \mu)_\alpha:=(\widetilde e; \mu)_\alpha$. 
\end{notation}

\begin{prop}[Kirchhoff formulas for weights]\label{prop:kirch}
Let $\alpha: F_N\rightarrow \pi_1(\Gamma)$ be a marking on $F_N$, let $X=\widetilde \Gamma$, and let $K$ be a finite non-degenerate subtree of $X$ with terminal edges $e_1,\dots, e_n$, as in Definition \ref{defn:gcyl}. Let $\mu\in \gcn$.
Then for every $i=1,\dots, n$ we have
\[
(K; \mu)_\alpha=\sum_{U\in P_+(q(e_i))} (K\cup U; \mu)_\alpha ,\tag{$\bigstar$}
\]
in notations of \ref{notation:edges}.
\end{prop}
\begin{proof}
The statement follows directly from Lemma~\ref{lem:disj} and from the fact that $\mu$ is finite-additive.
\end{proof}

Since the Borel $\sigma$-algebra on $\mathfrak C_N$ is generated by
the collection of all subset cylinders, it follows, by Kolmogorov Extension Theorem, that any $F_N$-invariant system of weights on all the subset cylinders satisfying the Kirchhoff formulas actually defines a subset current:

\begin{prop}\label{prop:weights}
Let $\alpha, \Gamma$ and $X$ be as in Proposition~\ref{prop:kirch}.
Let $\mathcal K_\Gamma$ be the set of all finite non-degenerate simplicial subtrees of $X$ and let 
\[
\vartheta: \mathcal K_\Gamma\to [0,\infty)
\]
be a function satisfying the following conditions:
\begin{enumerate}
\item For every $K\in \mathcal K_\Gamma$ and every $g\in F_N$ we have $\vartheta(gK)=\vartheta(K)$.

\item For every $K\in \mathcal K_\Gamma$ and every terminal edge $e$ of $K$ we have
\[
\vartheta(K)=\sum_{U\in P_+(q(e))} \vartheta(K\cup U).
\]

\end{enumerate}
Then there exists a unique $\mu\in \gcn$ such that for every $K\in \mathcal K_\Gamma$ we have
\[
\vartheta(K)=(K; \mu)_\alpha.
\]
\end{prop}

\begin{prop}
The space $\gcn$ is locally compact and the space $\pgcn$ is compact.
\end{prop}
\begin{proof}
It follows from the definition of $\gcn$ as the space of $F_N$-invariant locally finite positive Borel measures on $\mathfrak C_N$ that $\gcn$ is metrizable. Then to show local compactness it suffices to establish sequential local compactness of $\gcn$. 

Consider a marking $\alpha:F_N\isom \pi_1(\Gamma)$. It is not hard to show, using Proposition~\ref{prop:kirch},  Proposition~\ref{propdfn:weights} and a standard diagonalization argument for individual weights, that for every $C>0$ the sets
\[
\{\mu\in \gcn: \sum_{e\in E_{top}\Gamma} (\tilde e; \mu)_\alpha\le C \}
\]
and
\[
\{\mu\in \gcn: \sum_{e\in E_{top}\Gamma} (\tilde e; \mu)_\alpha= C \}
\]
are sequentially compact. This implies local compactness of $\gcn$ and compactness of $\pgcn$.
\end{proof}

\begin{rem}\label{rem:curr}
Recall also that elements of $\Curr(F_N)$ are positive locally finite $F_N$-invariant Borel measures on the space of $2$-element subsets in $\partial F_N$ which is clearly a closed $F_N$-invariant subset of $\mathfrak C_N$. 
Hence the space $\Curr(F_N)$ can be thought of as canonically embedded in $\gcn$, $\Curr(F_N)\subseteq \gcn$, and it is not hard to see that $\Curr(F_N)$ is a closed subset of $\gcn$. Moreover, once the action of $\Out(F_N)$ on $\gcn$ is defined in Section \ref{sec:action} below, it will be obvious that $\Curr(F_N)$ is an $\Out(F_N)$-invariant subset of $\gcn$. For similar reasons $\PCN\subseteq \pgcn$ is a closed $\Out(F_N)$-invariant subset.
\end{rem}

\section{Rational subset currents}\label{sec:rational}

\subsection{Counting and rational subset currents}\label{subsec:counting}

Recall that for a subgroup $H\le G$ of a group $G$ the
\emph{commensurator} or \emph{virtual normalizer} $Comm_G(H)$ of $H$ in $G$ is defined as
\[
Comm_G(H):=\{g\in G| [H: H\cap gHg^{-1}]<\infty, \text{ and } [gHg^{-1}: H\cap gHg^{-1}]<\infty\}.
\]
It is easy to see that $Comm_G(H)$ is again a subgroup of $G$ and that
$H\le Comm_G(H)$. 

Suppose now that $N\ge 2$ and $H\le F_N$ is a nontrivial subgroup of
$F_N$. The \emph{limit set} $\Lambda(H)$ of $H$ in $\partial F_N$ is
the set of all $\xi\in \partial F_N$ such that there exists a sequence
$h_n\in H$, $n\ge 1$ satisfying
\[
\lim_{n\to\infty} h_n=\xi \quad\text{ in } F_N\cup \partial F_N.
\]
We recall some elementary properties of limit sets. We see in particular that for every $H\le F_N, H\ne 1$ we have $\Lambda(H)\in
\mathfrak C_N$.

\begin{prop}\label{prop:basic}
Let $H\le F_N$ be a nontrivial subgroup. Then:
\begin{enumerate}
\item The limit set $\Lambda(H)$ is a closed  $H$-invariant subset of
  $\partial F_N$. If $H$ is infinite cyclic, $\Lambda(H)$ consists of two distinct points; if $H$ is not cyclic, $\Lambda(H)$ is infinite.
\item For any $\xi\in \Lambda(H)$ the closure of the orbit $H\xi$ in
  $\partial F_N$ is equal to $\Lambda(H)$.
\item For every $g\in G$
\[
\Lambda(gHg^{-1})=g\Lambda(H).
\]
\item If $H\le Q\le F_N$ then $\Lambda(H)\subseteq \Lambda(Q)$.
\item Either $H$ is infinite cyclic and $\Lambda(H)$ consists of
  exactly two distinct points or $H$ contains a nonabelian free
  subgroup and $\Lambda(H)$ is uncountable. 
\item Let $T\in \cvn$ and let $Conv_T(\Lambda(H))$ be the \emph{convex hull} of
  $\Lambda(H)$ in $T$, that is, the union of all bi-infinite geodesics
  in $T$ with endpoints in $\Lambda(H)$. Then  $Conv_T(\Lambda(H))=T_H$
  is the unique minimal $H$-invariant subtree of $T$.
\end{enumerate}
\end{prop}

Another useful basic fact relates limit sets of finitely
generated subgroups and their commensurators (see~\cite{KS,KM} for details).
\begin{prop}\label{prop:comm}
Let $H\le F_N$ be a nontrivial finitely generated subgroup. For a subset $Z\subseteq \partial F_N$ denote $Stab_{F_N}(Z):=\{g\in
F_N: gZ=Z\}$.
Then
\begin{enumerate}
\item $Stab_{F_N}(\Lambda(H))=Comm_{F_N}(H)$ and
  $[Comm_{F_N}(H):H]<\infty$.
\item $\Lambda(H)=\Lambda (Comm_{F_N}(H))$.
\item For $H_1=Comm_{F_N}(H)$ we have $Comm_{F_N}(H_1)=H_1$.
\item Let $L\le F_N$ such that $H\le L$. Then $[L:H]<\infty$ if and
  only if $L\le Comm_G(H)$.
\item Suppose $H=\langle g\rangle$, where $g\in F_N, g\ne 1$. Then
  $H=Comm_{F_N}(H)$ if and only if $g$ is not a proper power in $F_N$,
  that is, if and only if $H$ is a maximal infinite cyclic subgroup of
  $F_N$.
\end{enumerate}
\end{prop}

\begin{defn}\label{defn:count}
Let $H\le F_N$ be a nontrivial finitely generated subgroup. 

\begin{enumerate}
\item Suppose first that $H=Comm_{F_N}(H)$. Define the measure $\eta_H$ on $\mathfrak C_N$ as
\[
\eta_H:=\sum_{H_1\in [H]} \delta_{\Lambda(H_1)},
\]
where $[H]$ is the conjugacy class of $H$ in $F_N$.
\item Now let $H\le F_N$ be an arbitrary nontrivial finitely generated
  subgroup. Put $H_0:=Comm_{F_N}(H)$ and let $m:=[H_0:H]$.
  Proposition~\ref{prop:comm} implies that $m<\infty$ and that
  $H_0=Comm_{F_N}(H_0)$. Then put $\eta_H:=m\eta_{H_0}$.
\end{enumerate}
\end{defn}

\begin{lem}\label{lem:count}
Let $H\le F_N$ be a nontrivial finitely generated subgroup. Then
$\eta_H\in \gcn$.
\end{lem}
\begin{proof}
It is enough to show that $\eta_H\in \gcn$ for the case where
$H=Comm_{F_N}(H)$. Thus we assume that $H\le F_N$ is a nontrivial
finitely generated subgroup with $H=Comm_{F_N}(H)$. Fix a free basis
$A$ of $F_N$ and let $X$ be the Cayley graph of $F_N$ with respect to $A$.

Part~(3) of Proposition~\ref{prop:basic} implies that 
$\eta_H$ is an $F_N$-invariant positive Borel measure on $\mathfrak C_N$.
Thus it remains to check that $\eta_H(C)<\infty$ for every compact
$C\subseteq \mathfrak C_N$. Every compact subset of $\mathfrak C_N$ is
the union of finitely many cylinder subsets. Thus we only need to
check that $\eta_H(\mathcal SCyl_{X}(K))<\infty$ for every finite
non-degenerate subtree $K$ of $X$.  After replacing $K$ by its $F_N$-translate, we may assume that
the element $1\in F_N$ is a vertex of $K$. Also, since $\eta_H$
depends only on the conjugacy class of $H$, after replacing $H$ by a
conjugate we may assume that $1\in X_H=Conv_{X}(\Lambda(H))$.

Recall from Proposition \ref{prop:basic} that $Conv_X(\Lambda(H))=X_H$ is the unique minimal
$H$-invariant subtree of $X$. Whenever $g\in F_N$ is such that
$g\Lambda(H)=\Lambda(gHg^{-1})\in \mathcal SCyl_X(K)$, we have $1\in
Conv_X(\Lambda(H))=X_H$. 
Thus it suffices to show that the number of distinct translates $gX_H$
of $X_H$ that contain $1\in F_N$ is finite.

It is not hard to see, since by assumption $1\in X_H$, that for $g\in
F_N$ we have $1\in gX_H$ if and only if $g\in V(X_H)\subseteq F_N$.

Since $X_H$ is the minimal $H$-invariant subtree and $H$ is finitely
generated, the quotient $H\setminus X_H$ is a finite graph. In particular
$H\setminus V(X_H)$ is a finite set. Every $H$-orbit of a vertex of $X_H$ is a
coset class $uH$ for some $u\in F_N$. Thus there exists a finite
set $u_1,\dots, u_m\in F_N$ such that $V(X_H)=\{u_ih| h\in H, 1\le
i\le m\}$. For $g=u_ih$, where $h\in H$, $1\le i\le m$, we have
$gX_H=u_ihX_H=u_iX_H$. Thus indeed there are only finitely many
translates of $X_H$ that contain $1\in F_N$. Therefore
$\eta_H(\mathcal SCyl_X(K))<\infty$, as required. 

\end{proof}

\begin{defn}[Rational and counting currents]
For a nontrivial finitely generated $H\le F_N$, we call
$\eta_H\in\gcn$ given by Definition \ref{defn:count} and Lemma \ref{lem:count} the \emph{counting subset current}
associated to $H$.

A subset current $\mu\in\gcn$ is called \emph{rational} if
$\mu=r\eta_H$ for some $r\ge 0$ and some nontrivial finitely generated
subgroup $H\le F_N$.
\end{defn}

Definition~\ref{defn:count} directly implies:
\begin{prop}
Let $H\le F_N$ be a nontrivial finitely generated subgroup and let $H'=gHg^{-1}$ for some $g\in F_N$.
Then $\eta_H=\eta_{H'}$.
\end{prop}

The above statement has a partial converse:
\begin{prop}
Let $H,H'\le F_N$ be nontrivial finitely generated subgroups such that $H=Comm_{F_N}(H)$ and $H'=Comm_{F_N}(H')$.
Then $\eta_H=\eta_{H'}$ if and only if $[H]=[H']$. 
\end{prop}
\begin{proof}
We have already seen that if $[H]=[H']$ then $\eta_H=\eta_{H'}$.

Suppose now that $\eta_H=\eta_{H'}$.  Choose a marking $\alpha:F_N\to\pi_1(\Gamma)$ on $F_N$. Let $X=\widetilde \Gamma$.

Let $K\in\mathcal K_\Gamma$ be such that $\Lambda(H)\in \mathcal SCyl_\alpha(K)$. Since $\eta_H(\mathcal SCyl_\alpha(K))<\infty$ and $\eta_{H'}(\mathcal SCyl_\alpha(K))<\infty$,  Definition~\ref{defn:count} implies that only finitely many distinct $F_N$-translates of $\Lambda(H)$ and $\Lambda(H')$ belong to $\mathcal SCyl_\alpha(K)$.  Let these translates be $g_1\Lambda(H), \dots, g_m \Lambda(H)$ and $f_1\Lambda(H'), \dots, f_t\Lambda(H')$, with $g_1=1$. Thus we have $m+t$ distinguished points in the open set $\mathcal SCyl_\alpha(K)$. Since $\mathfrak C_N$ is metrizable and Hausdorff, we can find an open subset $V$ of $\mathcal SCyl_\alpha(K)$ such that $V$ contains exactly one of these $m+t$ points, namely the point $g_1\Lambda(H)=\Lambda(H)$. Since the cylinder subsets form a basis of open sets for $\mathfrak C_N$,  there exists a finite subtree $K_1$ of $X$  such that $\Lambda(H)\in \mathcal SCyl_\alpha(K_1)\subseteq V$. Thus  $\Lambda(H)\in \mathcal SCyl_\alpha(K_1)$ and no $F_N$-translate of $\Lambda(H), \Lambda(H')$, distinct from $\Lambda(H)$, belongs to $\mathcal SCyl_\alpha(K_1)$. Then $\eta_H(\mathcal SCyl_\alpha(K_1))=1$. By assumption $\eta_H=\eta_{H'}$ and hence $\eta_{H'}(\mathcal SCyl_\alpha(K_1))=1$.  Thus, by definition of $\eta_{H'}$,  there is a unique translate $g\Lambda(H')$, where $g\in F_N$, such that $g\Lambda(H')\in \mathcal SCyl_\alpha(K_1)$. By the choice of $K_1$, the only $F_N$-translate of $\Lambda(H), \Lambda(H')$ contained in $\mathcal SCyl_\alpha(K_1)$ is $\Lambda(H)$. Therefore $\Lambda(H)=g\Lambda(H')$ so that  $\Lambda(H)=g\Lambda(H')=\Lambda(gH'g^{-1})$. 

Since $H=Comm_{F_N}(H)$ and $gH'g^{-1}=Comm_{F_N}(gH'g^{-1})$, Proposition~\ref{prop:comm} implies that
\[
H=Stab_{F_N}(\Lambda(H))=Stab_{F_N}( \Lambda(gH'g^{-1}) )=gH'g^{-1},
\]
and $[H]=[H']$, as required.
\end{proof}

\subsection{$\Gamma$-graphs}\label{subsec:graphs}

We need to introduce some terminology slightly generalizing the
standard set-up of Stallings foldings for subgroups of free
groups (see \cite{Sta,KM}). That set-up usually involves choosing a free basis $A$ of
$F_N$ and considering graphs whose edges are labelled by elements of
$A^{\pm 1}$. Choosing a free basis $A$ of $F_N$ corresponds to a
marking on $F_N$ that identifies $F_N$ with the fundamental group of the standard $N$-rose $R_N$. We need to
relax the requirement that the graph in question be $R_N$.

\begin{defn}[$\Gamma$-graphs]\label{defn:gammagraph}
Let $\Gamma$ be a finite connected graph without degree-one and
degree-two vertices. A \emph{$\Gamma$-graph} is a graph $\Delta$
together with a graph morphism $\tau:\Delta\to\Gamma$.

For a vertex $x\in V\Delta$ we say that the \emph{type} of $x$ is the
vertex $\tau(x)\in V\Gamma$. Similarly, for an oriented edge $e\in
E\Gamma$ the \emph{type} of $e$, or the \emph{label} of $e$ is the
edge $\tau(e)$ of $\Gamma$.
\end{defn}

Note that every covering of $\Gamma$ has a canonical $\Gamma$-graph
structure. In particular, $\Gamma$ itself is a $\Gamma$-graph and so
is the universal cover $\widetilde\Gamma$ of $\Gamma$. Also, every
subgraph of a $\Gamma$-graph is again a $\Gamma$-graph.

Note that for a graph $\Delta$ a graph-morphism $\tau:\Delta\to\Gamma$
can be uniquely specified by assigning a label $\tau(e)\in E\Gamma$
for every $e\in E\Delta$ in such a way that
$\tau(e^{-1})=(\tau(e))^{-1}$ and such that for every pair of edges
$e',e'\in E\Delta$ with $o(e)=o(e')$ we have
$o(\tau(e))=o(\tau(e'))\in V\Gamma$. Thus we usually will think of a
$\Gamma$-graph structure on $\Delta$ as such an assignment of labels
$\tau: E\Delta\to E\Gamma$. Also, by abuse of notation, we will often
refer to a graph $\Delta$ as a $\Gamma$-graph, assuming that the
graph morphism $\tau:\Delta\to\Gamma$ is implicitly specified.


\begin{defn}[$\Gamma$-graph morphism]
Let $\tau_1:\Delta_1\to\Gamma$ and $\tau_2:\Delta_2\to\Gamma$ be $\Gamma$-graphs. A graph morphism
$f:\Delta_1\to\Delta_2$ is called a \emph{$\Gamma$-graph morphism}, if it respects the labels of vertices and
edges, that is if $\tau_1=\tau_2\circ f$.
\end{defn}

\begin{defn}[Folded $\Gamma$-graph]
Let $\Delta$ be a $\Gamma$-graph. For a vertex $x\in V\Delta$ denote
by $Lk_\Delta(x)$ (or just by $Lk(x)$) the {\it link} of $x$, namely the function
\[
Lk_\Delta(x): E\Gamma\to\mathbb Z_{\ge 0}
\]
where for every $e\in E\Gamma$ the value
$\left(Lk_\Delta(x)\right)(e)$ is the number of edges of $\Delta$ with
origin $x$ and label $e$.

A $\Gamma$-graph $\Delta$ is said to be \emph{folded} if for every
vertex $x\in V\Delta$ and every $e\in E\Gamma$ we have
\[
\left(Lk_\Delta(x)\right)(e)\le 1.
\]

If $\Delta$ is folded, we will also think of $Lk_\Delta(x)$ as a
subset of $E\Gamma$ consisting of all those $e\in E\Gamma$ with
$\left(Lk_\Delta(x)\right)(e)=1$, that is, of all $e\in E\Gamma$ such
that there is an edge in $\Delta$ with origin $x$ and label $e$.

\end{defn}

The following is an immediate corollary of the definitions.
\begin{lem}
Let $\Delta$ be a $\Gamma$-graph and let $\tau:\Delta\to\Gamma$ be the
associated graph morphism. Then:
\begin{enumerate}
\item $\Delta$ is folded if and only if $\tau$ is an immersion. 

\item Suppose $\Delta$ is connected. Then $\Delta$ is a covering
of $\Gamma$ if and only if $Lk_\Delta(x)=Lk_\Gamma(\tau(x))$ for every
$x\in V\Delta$.

\item Let $\Delta_1,\Delta_2$ be two $\Gamma$-graphs such that
  $\Delta_1$ is connected and $\Delta_2$ is folded. Let $x_1\in
  V\Delta_1$, $x_2\in V\Delta_2$ be vertices of the same type $x\in
  V\Gamma$. Then there exists at most one $\Gamma$-graph morphism
  $f:\Delta_1\to\Delta_2$ such that $f(x_1)=x_2$.
\end{enumerate}
\end{lem}

\begin{defn}[$\Gamma$-core graph]\label{defn:core}
Referring to Stallings' notion of {\it core graphs}, \cite{Sta}, we say that a finite $\Gamma$-graph $\Delta$ is a $\Gamma$-{\it core graph} if
$\Delta$ is \emph{cyclically reduced}, that is it is folded and has no degree-one and degree-zero
vertices.

\end{defn}

Informally, we think of  $\Gamma$- core graphs as
generalizations of cyclic words, over the \lq\lq alphabet\rq\rq\ $\Gamma$. We
need the following analog of the notion of a subword in this context:

\begin{defn}[Occurrence]\label{defn:occur}
Let $K\subseteq \widetilde \Gamma$ be a finite non-degenerate subtree, considered together with the
canonical $\Gamma$-graph structure inherited from $\widetilde \Gamma$.
Let $\Delta$ be a finite $\Gamma$-core graph. 
An \emph{occurrence} of $K$ in $\Delta$ is a $\Gamma$-graph morphism
$f:K\to\Delta$ such that for every vertex $x$ of $K$ of degree at
least $2$ in $K$ we have $Lk_K(x)=Lk_\Delta(f(x))$.

We denote the number of all occurrences of $K$ in $\Delta$ by $(K; \Delta)_\Gamma$, or just $(
K; \Delta)$.
\end{defn}

In topological terms, a $\Gamma$-graph morphism $f:K\to\Delta$ is an occurrence
of $K$ in $\Delta$ if $f$ is an immersion and if $f$ is a covering map
at every point $x\in K$ (including interior points of edges) except
for the degree-1 vertices of $K$. That is, for every $x\in K$, other
than a degree-1 vertex of $K$, $f$ maps a small neighborhood of $x$ in
$K$ homeomorphically \emph{onto} a small neighborhood of $f(x)$ in
$\Delta$.

\begin{prop}\label{prop:wd}
Let $\alpha:F_N\to\pi_1(\Gamma)$ be a marking on $F_N$ and let
$X=\widetilde\Gamma$. As before (see \ref{prop:weights}), let $\mathcal K_\Gamma$ denote the set of
all non-degenerate finite simplicial subtrees of $X$.
Let $\Delta$ be a finite $\Gamma$-core graph.
Define $\vartheta_\Delta: \mathcal K_\Gamma\to \mathbb R$ as
\[
\vartheta_\Delta(K):=(K;\Delta)_\Gamma
\]
for each $K\in \mathcal K_\Gamma$. 

Then the function
$\vartheta_\Delta$ satisfies the conditions of
Proposition~\ref{prop:weights}. Thus there is a unique subset
current $\mu_\Delta\in\gcn$ such that for every $K\in \mathcal
K_\Delta$
\[
( K; \mu_\Delta)_\alpha=( K;\Delta)_\Gamma
\]
\end{prop}
\begin{proof}
We will see later on that, for a connected $\Delta$, the function $\vartheta_\Delta$ defines the counting current of
a finitely generated subgroup of $F_N$. However, we think it is useful
to have a direct argument for Proposition~\ref{prop:wd}.

It easily follows from the definitions that for every $K\in \mathcal K_\Gamma$ and every $g\in F_N$ we have $(K;\Delta)_\Gamma=(gK;\Delta)_\Gamma$.
Thus we only need to verify that the condition (2) of
Proposition~\ref{prop:weights} holds for $\vartheta_\Delta$.

Let $K\in \mathcal K_\Gamma$ and let $e$ be a terminal edge of
$K$, as in Definition \ref{defn:gcyl}. Consider the set of trees
\[
Q(K,e)=\{ K\cup U | U\in P_+(q(e))\},
\]
as in Notation~\ref{notation:edges}.
Thus each element of $Q(K,e)$ is obtained by adding to $K$ a nonempty
subset $U$ from the set of edges $q(e)$.

We now construct a function $D$ from the set $R$ of all occurrences of
$K$ in $\Delta$ to the set $R'$ of all occurrences of elements of
$Q(K,e)$ in $\Delta$. Put $x:=t(e)\in VK$.

Let $f:K\to\Delta$ be an occurrence of $K$. Since $\Delta$ is a core graph, the vertex $y=f(x)$ has degree bigger than 1 in
$\Delta$. Note that $f(e^{-1})$ is an edge of $\Delta$ with initial
vertex $y$. Put $U_0$ to be the set of the labels of all the edges in
$\Delta$ with origin $y$, excluding the edge $f(e^{-1})$. Thus $U_0$ is
nonempty. Let $U$ be the set of all edges in $\widetilde \Gamma$ with
initial vertex $x$ and with label belonging to $U_0$. Put $K'=K\cup U$,
so that $K'\in Q(K,e)$. We now extend $f$ to a morphism $f':K'\to\Delta$ by
sending every edge from $U$ to the unique edge in $\Delta$ with the
same label and with origin $y$. Then by construction, $f'$ is a
$\Gamma$-graph morphism whose restriction to $K$ is $f$ and, moreover,
$Lk_{K'}(x)=Lk_{\Delta}(y)$. Therefore $f'$ is an occurrence of $K'$
in $\Delta$. We put $D(f):=f'$.

This defines a function $D:R\to R'$. By construction, this function is
injective. Moreover, $D$ is clearly onto. Indeed, if $K'=K\cup U\in
Q(K,e)$ and $f':K'\to \Delta$ is an occurrence of $K'$ in $\Delta$
then $f:=f'|_K$ is an occurrence of $K$ in $\Delta$ and $D(f)=f'$.
Thus $D$ is a bijection, and
hence $R$ and $R'$ have equal cardinalities. It follows that
\[
( K; \Delta)_\Gamma=\sum_{U\in P_+(q(e))} (K\cup U; \Delta)_\Gamma,
\]
as required.
\end{proof}

Note that if $\Delta$ is a finite $\Gamma$-core graph
with connected components $\Delta_1, \dots, \Delta_k$, then each $\Delta_i$ is again a $\Gamma$-core graph and we have
\[
\mu_{\Delta}=\mu_{\Delta_1}+\dots+\mu_{\Delta_k}.
\]

We leave it to the reader to verify the following:
\begin{lem}\label{lem:m}
Let $\Delta$ be as in Proposition~\ref{prop:wd} and assume that
$\Delta$ is connected. Let
$\widehat\Delta$ be an $m$-fold cover of $\Delta$ for some $m\ge
1$. Then $\mu_{\widehat\Delta}=m \mu_\Delta$
\end{lem}

\subsection{Counting currents and $\Gamma$-graphs}\label{GraphsCurrents}

We now want to relate the current constructed in
Proposition~\ref{prop:wd} to counting currents of finitely generated
subgroups of $F_N$.



Let $\alpha: F_N\to\pi_1(\Gamma)$ be a marking on $F_N$, and let $X:=\widetilde \Gamma$. Recall from Subsection \ref{subsec:mark} that $\partial F_N$ and $\partial X$ are identified via the homeomorphism induced by $\alpha$.
Note that if $S\in \mathfrak C_N$, the fact that $S$ has cardinality at least
$2$ implies that its convex hull $Conv_X(S)$ is nonempty. Moreover, it is easy to see
that $Conv_X(S)$ is an infinite subtree of $X$ without any degree-1
vertices. Let us denote by $\mathcal T(X)$ the set of all infinite subtrees of $X$ without
degree-1 vertices.

The following statement is an elementary consequence of the
definitions and we leave the details to the reader.

\begin{prop}\label{prop:ch}
Let $\alpha, \Gamma, X$ be as above. Then
\begin{enumerate}
\item For any $S\in \mathfrak C_N$ we have $\partial Conv_X(S)=S$ and for any
  $Y\in \mathcal T(X)$ we have $Y=Conv_X(\partial Y)$. 
\item The convex hull operation yields a bijection $Conv_X: \mathfrak C_N\to \mathcal T(X)$ which is $F_N$-equivariant: for any $S\in \mathfrak C_N$ and $g\in F_N$ we have $Conv_X(gS)=gConv_X(S)$.
\item For any $S\in \mathfrak C$ we have $\partial Conv_X(S)=S$ and for any $Y\in \mathcal T(X)$ we have $Y=Conv_X(\partial Y)$. 
\end{enumerate}
\end{prop}

Recall that if $T\in \cvnbar$ and if $H\le F_N$ is a nontrivial
subgroup, there is a unique smallest $H$-invariant subtree of $T$
denoted $T_H$, and, moreover $T_H$ has no degree-one vertices.

Propositions~\ref{prop:ch}, \ref{prop:basic} and \ref{prop:comm} easily imply:
\begin{prop}\label{prop:XH}
Let $\alpha, \Gamma, X$ be as above. Let $H\le F_N$ be a nontrivial finitely
generated subgroup and let $X_H$ be the minimal $H$-invariant subtree
of $X$. Then:
\begin{enumerate}
\item $X_H\in\mathcal T(X)$ and $X_H=Conv_X(\Lambda(H))$. 
\item For any $g\in F_N$, $gX_H=X_{gHg^{-1}}$.
\item $H=Comm_{F_N}(H)$ if and only if $Stab_{F_N}(X_H)=H$.
\end{enumerate}
\end{prop}


\begin{conv}[$\Gamma$-core graphs representing conjugacy classes of subgroups of $F_N$]\label{conv:subgroupcore}
Let $\alpha: F_N\to\pi_1(\Gamma)$ be a marking on $F_N$, and let $X:=\widetilde \Gamma$. Let $H\le F_N$ be a finitely generated
subgroup and let $X_H$ be the minimal $H$-invariant subtree of $X$. Note
that, being a subgraph of $\widetilde \Gamma$, the tree $X_H$ comes
equipped with a canonical $\Gamma$-graph structure $X_H\to\Gamma$ (which is just the restriction to $X_H$ of the
universal covering map $\widetilde \Gamma\to\Gamma$) which is
$F_N$-invariant. Put $\Delta=H\setminus X_H$. Then $\Delta$ is a
finite connected graph without degree-1 vertices (by minimality of
$X_H$); that is, $\Delta$ is a $\Gamma$-core graph. Moreover, $\Delta$ inherits a natural $\Gamma$-graph structure
from $X_H$.

It is easy to see, in view of part (2) of Proposition~\ref{prop:XH}, that the isomorphism type of $\Delta$ as a
$\Gamma$-graph depends only on the conjugacy class $[H]$ of $H$ in
$F_N$. Thus we say that $\Delta$ is the $\Gamma$-{\it core graph representing $[H]$}.
\end{conv}


We can also describe $\Delta$ as follows: $\Delta$ is the
\emph{core} of $\widehat \Gamma = H\setminus X$, that is, $\Delta$ is the union of
all immersed (but not necessarily simple) circuits in $\widehat \Gamma$. Thus $\Delta$ and
$\widehat \Gamma$ are homotopy equivalent, and $\Delta$ is obtained
from $\widehat \Gamma$ by \lq\lq cutting-off\rq\rq\ a (possibly empty) collection
of infinite tree \lq\lq hanging branches\rq\rq . Also, $\Delta$ is the smallest subgraph
of $\widehat \Gamma$ whose inclusion in $\widehat \Gamma$ is a
homotopy equivalence with $\widehat\Gamma$. The labelling map $\tau:\Delta\to\Gamma$ satisfies
$\tau_\#(\pi_1(\Delta))=H$. In fact $\tau_\#$ is an isomorphism between $\pi_1(\Delta)$ and $H$ and sometimes,
by abuse of notation, we will write $\pi_1(\Delta)=H$.


\begin{thm}\label{thm:cc}
Let $\alpha: F_N\to\pi_1(\Gamma)$ be a marking on $F_N$. Let $H\le
F_N$ be a nontrivial finitely generated subgroup such that $Comm_{F_N}(H)=H$. Let $\Delta$ be the $\Gamma$-core graph representing $[H]$. Then $\eta_H=\mu_\Delta$.
\end{thm}
\begin{proof}
We need to show that for every finite non-degenerate subtree $K$ of
$X=\widetilde \Gamma$ we have $(K; \eta_H)_\alpha=(K;\mu_\Delta)_\alpha$.

Choose a vertex $x_0$ in $K$ and let $v_0$ be the projection of $x_0$
in $\Gamma$. We may assume that $v_0$ is the base-point of $\Gamma$
and that $\alpha:F_N\isom \pi_1(\Gamma,v_0)$. Recall that $X_H=Conv_X(\Lambda
H)$ is the smallest $H$-invariant subtree of $X$. By replacing $H$
by its conjugate if necessary, we may assume that $x_0\in
X_H$. Moreover, $\Delta=H\setminus X_H$, and $(X_H,x_0)$ is canonically identified with $\widetilde
{(\Delta,y_0)}$, where $y_0$ is the image of $x_0$ under the projection $p:X_H\rightarrow\Delta$.  

We need to show that the number $( K;\Delta)_\alpha$ of
occurrences of $K$ in $\Delta$ is equal to the number of distinct
$F_N$-translates of $X_H$ that contain $K$. We will construct a function $Q$ from the set of occurrences of $K$ in
$\Delta$ to the set of $F_N$-translates of $X_H$ that contain $K$.

Let $f:K\to \Delta$ be an occurrence of $K$ in $\Delta$. Choose an edge-path
$\gamma$ from $y_0$ to $f(x_0)$ in $\Delta$. The fact that both $y_0$
and $f(x_0)$ project to $v_0$ in $\Gamma$ implies that the (unique) lift
$\widetilde\gamma$ of $\gamma$ to an edge-path in $X$ with origin $x_0$ has
its terminal vertex of the form $gx_0$ for some (unique) $g\in F_N$. Then
the definition of occurrence of $K$ in $\Delta$ and the fact that
$\widetilde \Delta$ is identified with $X_H$ imply that $gK\subseteq
X_H$ and so $K\subseteq g^{-1}X_H$. We set $Q(f):=g^{-1}X_H$. 
We need to check that $Q(f)$ is well-defined, that is, that the translate
$g^{-1}X_H$ does not depend on the particular choice of an edge-path
$\gamma$ from $y_0$ to $f(x_0)$ in $\Delta$. Indeed, if $\gamma'$ is
another such edge-path, then $\gamma'\gamma^{-1}\in
\pi_1(\Delta,y_0)=H$. Hence if $\gamma'$ lifts to an edge-path from $x_0$ to
$g'x_0$ in $X$, we have $g'g^{-1}\in H$, so that $g^{-1}=(g')^{-1}h$
for some $h\in H$ and hence $g^{-1}X_H=(g')^{-1}X_H$. 
Thus the translate $g^{-1}X_H$ containing $K$ is uniquely determined
by the occurrence $f:K\to \Delta$ of $K$ in $\Delta$, so that $Q(f)$
is well-defined. Hence we have
constructed a map $Q$ from the set of occurrences of $K$ in $\Delta$
to the set of $F_N$-translates of $X_H$ containing $K$.

We claim that the map $Q$ is injective. Indeed, let $f_1:K\to\Delta$ be another occurrence of $K$ in
$\Delta$. Let
$\gamma_1$ be an edge-path in $\Delta$ from $y_0$ to $f_1(x_0)$ and let
$g_1\in F_N$ be such that the lift to $X$ of $\gamma_1$ with origin
$x_0$ has terminus $g_1x_0$. Suppose that $g_1^{-1}X_H=g^{-1}X_H$. Then
$g_1g^{-1}X_H=X_H$ and hence $g_1g^{-1}\in Comm_{F_N}(H)$. By
assumption, on $H=Comm_{F_N}(H)$, so that $g_1g^{-1}\in H$
and $g_1=hg$ for some $h\in H$. 

However, since $\pi_1(\Delta,y_0)=H$, it follows that $h$ labels a
closed edge-path from $y_0$ to $y_0$ in $\Delta$ and hence, up to edge-path
reductions, $\gamma_1=\beta\gamma$ for some closed loop $\beta$ from
$y_0$ to $y_0$ in $\Delta$. Therefore the termini of $\gamma$ and
$\gamma_1$ are the same, that is 
$f(x_0)=f_1(x_0)\in V\Delta$. It follows, since all the graphs under
consideration are folded, that $f=f_1$. Thus indeed, the map $Q$ is
injective. 

We claim that $Q$ is also surjective. Suppose a translate
$g^{-1}X_H$ contains $K$. Then $X_H$ contains $gK$. Both $x_0$
and $gx_0$ belong to $X_H$ and hence the geodesic edge-path $[x_0,gx_0]$ is
contained in $X_H$. This edge-path projects to an edge-path $\gamma$ in $\Delta=
H\setminus X_H$ from $y_0$ to the vertex $z_0$ which is the image of
$gx_0$ in $\Delta$ under the projection $p: X_H\to \Delta$. Since $p$ is a
covering map, $f: K\to \Delta$ defined as $f(x)=p(g(x))$, $x\in
K$, is an occurrence of $K$ in $\Delta$, with $f(x_0)=z_0$. The definition of $Q$ now
gives $Q(f)=g^{-1}X_H$. 
 
Thus $Q$ is a bijection between the set of
occurrences of $K$ in $\Delta$ and the set of $F_N$-translates of
$X_H$ containing $K$. It follows that $( K; \eta_H)_\alpha=(K; \mu_\Delta)_\alpha$, as required.
\end{proof}
 
\begin{cor}\label{cor:cc}
Let $H\le F_N$ be any nontrivial finitely generated subgroup. Let $\alpha:F_N\to\Gamma$ be a marking and let $\Delta$ be the $\Gamma$-core graph 

representing $[H]$. Then $\eta_H=\mu_\Delta$.
\end{cor} 

\begin{proof}
Put $H_1=Comm_{F_N}(H)$. Then by Proposition~\ref{prop:comm} $H_1=Comm_{F_N}(H_1)$, and $[H_1:H]=m<\infty$. 
Let $\Delta_1$ be the $\Gamma$-core graph representing $[H_1]$.  Therefore by Theorem~\ref{thm:cc} $\eta_{H_1}=\mu_{\Delta_1}$. 
Moreover,  $\Delta$ is an $m$-fold cover of $\Delta_1$ and hence by Lemma~\ref{lem:m} $\mu_\Delta=m \mu_{\Delta_1}$. Also, by definition, $\eta_H=m\eta_{H_1}$. Hence $\eta_H=m\eta_{H_1}=m\mu_{\Delta_1}=\mu_\Delta$, as required.
\end{proof}

\section{Rational currents are dense}\label{sec:dense}

\begin{conv}\label{conv:GA} In order to avoid technical complications in the proof of the main result of this Section, Theorem \ref{thm:dense} in Subsection \ref{subsec:result} below, we will restrict ourselves in this section to only considering $\Gamma$-graphs for $\Gamma=R_N$, the standard $N$-rose. The universal cover $X=\widetilde R_N$ is then the Cayley graph of $F_N$ with respect to some basis. We will fix some basis $A$ in $F_N$ and we will think of $A$ as defining a marking $\alpha_A:F_N\to R_N$, that will also be fixed for the remainder of this Section.


Note that an $R_N$-structure on a graph $\Delta$ can be specified by
assigning every edge $e\in E\Delta$ a label $\tau(e)\in A^{\pm 1}$ so
that $\tau(e^{-1})=(\tau(e))^{-1}$.

   \end{conv}


\subsection{Linear span of rational currents}\label{subsec:linearspan}

\begin{prop}\label{prop:span}  Denote by $\mathcal SCurr_r(F_N)$ the set of all rational subset currents. 
Let $Span\left(\mathcal SCurr_r(F_N) \right)$ denote the $\mathbb R_{\ge 0}$-linear span of $\mathcal SCurr_r(F_N)$.

Let $N\ge 2$. Then the set  $\mathcal SCurr_r(F_N)$ is a dense subset of $Span\left(\mathcal SCurr_r(F_N) \right)$.
\end{prop}

\begin{proof}
We need to show that an arbitrary linear combination $c_1\eta_{H_1}+\dots +c_k\eta_{H_k}$ (where $k\ge 1$, $c_i\in \mathbb R_{\ge 0}$) can be approximated by currents of the form $c\eta_H$, where $c\in \mathbb R_{\ge 0}$.  Arguing by induction on $k$, we see that it suffices to prove this statement for $k=2$.

Every current of the form $c_1\eta_{H_1}+c_2\eta_{H_2}$, where $c_1,c_2\in \mathbb R_{\ge 0}$ can be approximated by currents of the form $r_1\eta_{H_1}+r_2\eta_{H_2}$, where $r_2,r_2$ are positive rational numbers. Thus it suffices to approximate by rational currents every current of the form $r_1\eta_{H_1}+r_2\eta_{H_2}$ where $r_1,r_2>0$ are rational numbers. Taking $r_1,r_2$ to a common denominator, we have $r_1=p_1/q$, $r_2=p_2/q$ where $p_1,p_2,q>0$ are integers. Since dividing a rational current by $q$ again yields a rational current,  it is enough to approximate  by rational currents all currents of the form $p_1\eta_{H_1}+p_2\eta_{H_2}$ where $p_1,p_2$ are positive integers. However, $p_1\eta_{H_1}=\eta_{L_1}$, $p_2\eta_{H_2}=\eta_{L_2}$, where $L_i$ is a subgroup of index $p_i$ in $H_i$ for $i=1,2$. Thus we only need to show that the sum of any two counting currents can be approximated by rational currents.

Suppose now that $H_1, H_2\le F_N$ are two nontrivial finitely
generated subgroups and let $\mu=\eta_{H_1}+\eta_{H_2}$.  
For $i=1,2$ let $\Delta_i$ be the $R_N$-core graph representing $[H_i]$, as in Convention \ref{conv:subgroupcore}. Thus, by Theorem~\ref{thm:cc},  $\eta_{H_i}=\mu_{\Delta_i}$ for $i=1,2$.

For $n=1,2,\dots $ let $\Lambda_{i,n}$ be a connected $n$-fold cover of $\Delta_i$.   
We now define a sequence of finite connected $R_N$-core graphs as follows.
Let $n\ge 1$. First assume that each of $\Lambda_{1,n}$, $\Lambda_{2,n}$ has a vertex $v_{i,n}$ of degree $<2N$, where $i=1,2$. Then, since $A^{\pm 1}$ contains at least 4 distinct letters,  there exists an $R_N$-graph $[u_1,u_2]$, which is a segment of three edges with origin denoted $u_1$ and terminus denoted $u_2$, such that identifying the origin of this segment with $v_{1,n}$ and the terminus of this segment with $v_{2,n}$ yields a folded $R_N$-graph:
\[
\Lambda_n:=\Lambda_{1,n}\cup \Lambda_{2,n}\cup [u_1,u_2] / \sim
\]
where  $u_1\sim v_{1,n}$, $u_2\sim v_{2,n}$. Note that by construction $\Lambda_n$ is connected and cyclically reduced.

Suppose now that $\Lambda_{1,n}$ has a vertex $v_{1,n}$ of degree $<2N$ but that every vertex in $\Lambda_{2,n}$ has degree $2N$.  

Let $v_{2,n}$ be any vertex of  $\Lambda_{2,n}$. Let $\Lambda_{2,n}'$ be obtained from $\Lambda_{2,n}$ by removing one of the edges incident to $v_{2,n}$.  Note that $\Lambda_{2,n}'$ is still finite, connected and cyclically reduced.

Then, as in the previous case,  there exists  a simplicial segment $[u_1,u_2]$ consisting of three edges labelled by elements of $A^{\pm 1}$ such that 
\[
\Lambda_n:=\Lambda_{1,n}\cup \Lambda_{2,n}' \cup [u_1,u_2] / \sim
\]
is a folded $R_N$-graph, where $u_1\sim v_{1,n}$, $u_2\sim v_{2,n}$.
Again, by construction, $\Lambda_n$ is connected and cyclically reduced.
In the case where $\Lambda_{2,n}$ has a vertex  of degree $<2N$ but that every vertex in $\Lambda_{1,n}$ has degree $2N$,  we define $\Lambda_n$ in a similar way.

Suppose now that for $i=1,2$ every vertex of $\Lambda_{i,n}$ has degree $2N$. Then $\Lambda_{i,n}$ is a finite cover of $R_N$, so that $\Delta_1,\Delta_2$ are both finite covers of $R_N$. Let $m_i\ge 1$ be such that $\Delta_i$ is an $m_i$-fold cover of $R_N$. Choose $\Lambda_n$ to be any connected $n(m_1+m_2)$-fold cover of $R_N$.  

This defines the sequence $\Lambda_{n}$, $n=1,2,\dots$ of finite, connected $R_N$-core graphs. We claim that
\[
\lim_{n\to\infty} \frac{1}{n} \mu_{\Lambda_n} =\mu_{\Delta_1}+\mu_{\Delta_2} \text{  in  } \gcn.\tag{!}
\] 
We need to show that for every finite non-degenerate subtree $K$ of $X$ we have 
\[
\lim_{n\to\infty} \frac{1}{n} (K; \Lambda_n)= (K; \Delta_1) +( K; \Delta_2).\tag{!!}
\]
Let us fix such a subtree $K$.  Let $n\ge 1$ be arbitrary. If $\Lambda_n$ is of the last described type, where both $\Lambda_{1,n}$ and $\Lambda_{2,n}$ are finite covers of $R_N$, then by construction $\mu_{\Lambda_n}=n(\mu_{\Delta_1}+\mu_{\Delta_2})$, and (!!) obviously holds.
Suppose now that one of the other cases in the construction of $\Lambda_n$ occurs.
Since the graph $\Lambda_n$ is folded, and has a vertex degree $\le 2N$,  all but a bounded number (in terms of some constant depending on $K$ but independent of $n$), occurrences of $K$ in $\Lambda_n$ are disjoint from the segment $[v_{1,n}, v_{2,n}]$ and thus come from occurrences of $K$ in $\Lambda_{1,n}\sqcup \Lambda_{2,n}$.  Hence 
\[
\big|  (K; \Lambda_n) - ( K; \Lambda_{1,n}) - ( K; \Lambda_{2,n})  \big| \le C_K
\]
for some constant $C_K$ independent of $n$. Since $\Lambda_{i,n}$ is an $n$-fold cover of $\Delta_i$, it follows that
\[
\big|  ( K; \Lambda_n) - n(K; \Delta_{1}) - n(K; \Delta_{2})  \big| \le C_K.
\]
Dividing the above inequality by $n$ and passing to the limit as $n\to\infty$, we get (!!), as required. This implies (!).
By Theorem~\ref{thm:cc} $\mu_{\Lambda_n}=\eta_{L_n}$ for some finitely generated nontrivial $L_n\le F_N$.
Hence (!) implies that
\[
\lim_{n\to\infty} \frac{1}{n} \eta_{L_n} =\eta_{H_1}+\eta_{H_2} \text{  in  } \gcn,
\]
which completes the proof.
\end{proof}

\subsection{Subset currents as measures on the space of rooted trees}\label{subsec:root}

Recall that $\mathcal T(X)$ denotes the set of infinite subtrees of $X$ without degree-1 vertices.
Denote by $\mathcal T_1(X)$ the set of all those $T\in \mathcal T(X)$ that contain the vertex $1\in F_N$. Thus all elements of  $\mathcal T_1(X)$ come equipped with a base-point (or root), namely $1\in F_N$. Note also that since every element  $Y\in \mathcal T_1(X)$ is folded,  $Y$ does not admit any nontrivial $\Gamma$-graph automorphisms that fix the root vertex $1$.  Hence we can also think of $\mathcal T_1(X)$ as the set of rooted $R_N$-graphs that are folded infinite trees without degree-1 vertices.

We equip $\mathcal T_1(X)$ with local topology, namely we say that $Y, Y'\in \mathcal T_1(X)$  are close if for some large $n\ge 1$ we have $Y\cap B_X(n)=Y'\cap B_X(n)$, where $B_X(n)$ is the ball of radius $n$ in $X$. Equivalently, $Y$ and $Y'$ are close if $\partial Y, \partial Y'\subseteq \partial X$ have small Hausdorff distance as closed subsets of $\partial X$. This makes $\mathcal T_1(X)$ a compact totally disconnected topological space.

Recall that for any finite non-degenerate subtree $K$ of $X$ we have defined a subset cylinder set $\mathcal SCyl_\alpha(K)\subseteq \mathfrak C_N$, Definition \ref{defn:gcyl}. 
For every such $K$ containing $1\in F_N$, put $\mathcal T Cyl_\alpha(K)$ to be the set of all $Y\in \mathcal T_1(X)$ such that $\partial Y\in \mathcal SCyl_\alpha(K)$.  Thus $\mathcal T Cyl_\alpha(K)$ consists of all $Y\in \mathcal T_1(X)$ such that $K\subseteq Y$ and such that whenever $\xi\in \partial Y$,  then $[1,\xi)\cap K=[1,v]$ where $v$ is a vertex of degree one in $K$. The sets $\mathcal T Cyl_\alpha(K)$ are compact and open, and form a basis of open sets for the topology on $\mathcal T_1(X)$, when $K$ varies over all finite non-degenerate subtrees of $X$ containing $1\in F_N$.

The space $\mathcal T_1(X)$ has a natural (partially defined) root-change operation. For $g\in F_N$ denote by $\mathcal T_{1,g}(X)$ the set of all $Y\in \mathcal T_1(X)$ that contain the vertex $g\in F_N$.
Define $r_g: \mathcal T_{1,g}(X)\to\mathcal T_1(X)$ by $r_g(Y):=g^{-1}Y$ for $Y\in \mathcal T_{1,g}(X)$. In other words, $r_g$ moves the root vertex from $1$ to $g$ in $Y$.
Denote by $\mathcal M_1(X)$ the set of all finite positive Borel measures on $\mathcal T_1(X)$ that are invariant under all $r_g, g\in F_N$.

\begin{prop}\label{prop:m1}
There is a canonical $\mathbb R_{\ge 0}$-linear homeomorphism $\mathbf t: \mathcal M_1(X) \to \gcn$.
\end{prop}
\begin{proof}
Let $\mu\in  \mathcal M_1(X)$. Define a measure $\mathbf t(\mu)$ on $\mathfrak C_N$ as follows. For a finite non-degenerate subtree $K$ of $X$ containing $1$ put
\[
( [K]; \mathbf t(\mu))_\alpha:=\mu\left( \mathcal T Cyl_\alpha(K) \right).\tag{$\clubsuit$}
\]
Invariance of $\mu$ with respect to the root-change implies that if $K, K'$ are two finite non-degenerate subtrees of $X$ that contain $1$ and such that $K'=gK$ for some $g\in F_N$, then $\mu\left( \mathcal T Cyl_\alpha(K) \right)=\mu\left( \mathcal T Cyl_\alpha(K') \right)$, so that $( [K]; \mathbf t(\mu))_\alpha$ is well-defined.
  
The assumption that $\mu$ is invariant with respect to $\{r_g: g\in F_N\}$ translates into the fact that $(\clubsuit)$ defines a collection of weights satisfying the requirements of Proposition~\ref{prop:weights} so that $\mathbf t(\mu)$ is indeed a subset current. It is then easy to check that $\mathbf t$ is an $\mathbb R_{\ge 0}$-linear homeomorphism and we leave the details to the reader.
\end{proof}

\subsection{Weak approximation by finite graphs}\label{subsec:approx}

Consider the set $U_{N}$ of all $R_N$-graphs with vertices of degree $\le 2N$. For every integer $R\ge 1$ denote by $U_{N,R}\subset U_N$ the
set of all rooted $R$-balls in graphs from $U_{N}$.
Finally, denote by $U_{N,R}^f$ the set of all $K\in U_{N,R}$ such that $K$ is a folded tree, where the root
vertex is not of degree 1 and such that the distance from the
root to every vertex of degree 1 of $K$ is equal to $R$.


Observe that every tree $K\in U_{N,R}^f$ admits a unique (and injective) $R_N$-graph morphism to
$X$ which sends the root in $K$ to the vertex $1\in VX$. Thus we can
think of trees $K\in U_{N,R}^f$ as subtrees of $X$ with root $1$.

For $R\ge 1$, $\Upsilon\in U_{N,R}$ and an $R_N$-graph
$\Delta$ denote by $J(\Upsilon,R,\Delta)$ the set of all those vertices $v$ in $\Delta$
such that the $R$-ball centered at $v$ in $\Delta$ is isomorphic as
a rooted $R_N$-graph to $\Upsilon$.

We say that a subset current $\mu\in \gcn$ is \emph{normalized} if the corresponding measure $\mu' = (\mathbf t)^{-1}(\mu)\in  \mathcal M_1(X)$ is of total mass $1$, $\mu'(\mathcal T_1(X))=1$.  Note that if $\mu\in \gcn$ is a nonzero subset current and $\mu'=(\mathbf t)^{-1}(\mu)\in \mathcal M_1(X)$ then for $c=\mu'(\mathcal T_1(X))$ the current $\frac{1}{c}\mu\in\gcn$ is normalized.

The following definition is an adaptation, to our notations and to our context, of the definition of random weak limit of finite rooted graphs introduced by Benjamini and Schramm, \cite{BS}.

\begin{defn}[Weak approximation]\label{defn:wa} 
Let $\mu\in \gcn$ be a normalized subset current. We say that a sequence of
$R_N$-graphs $\Delta_n$ \emph{weakly approximates} $\mu\in
\gcn$ if the
following conditions hold:
\begin{enumerate}
\item $\Delta_n\in U_{N}$ for every $n\ge 1$.
\item $\#V\Delta_n\to\infty$ as $n\to\infty$.
\item For every $R\ge 1$ and every $K\in U_{N,R}^f$ we have
\[
\lim_{n\to\infty} \frac{\#J(K,R,\Delta_n)}{\#V\Delta_n}=(K;\mu).
\]
\item For every 
$R\ge 1$ and every $\Upsilon\in U_{N,R}$ such that $\Upsilon\not\in U_{N,R}^f$ we have
\[
\lim_{n\to\infty} \frac{\#J(\Upsilon,R,\Delta_n)}{\#V\Delta_n}=0.
\]
\end{enumerate}
In this case we write
\[
\Delta_n \xrightarrow{\rm w.a.} \mu.
\]
\end{defn}

\begin{lem}\label{lem:wa}
Let $\mu\in \gcn$ be a normalized current and let $\Delta_n$ be a sequence of
finite connected $R_N$-core graphs such that
$\Delta_n \xrightarrow{\rm w.a.} \mu$.
Then
\[
\lim_{n\to\infty} \frac{1}{\# V\Delta_n} \mu_{\Delta_n}=\mu \text{  in
   } \gcn.
\]
\end{lem}
\begin{proof}

Let $R\ge 1$ and let $K\in U_{N,R}^f$. Recall that we think of $K$
as a subtree of $X$ with the root of $K$ being the vertex $1\in VX$.
Denote by $(K;\Delta_n)_{bad}$ the number of all
occurrences $f:K\to\Delta_n$ such that $f$ is not an embedding. Then
the $R$-ball in $\Delta_n$ around the $f$-image of the base-point of
$K$ is not a tree. Since the set $U_{N,R}$ is finite and $\Delta_n \xrightarrow{\rm w.a.} \mu$, the definition
of weak approximation then implies that
\[
\lim_{n\to\infty} \frac{(K;\Delta_n)_{bad}}{\#V\Delta_n}=0.
\]
We also have $(K;\Delta_n)=\#J(K,R,\Delta_n)+(
K;\Delta_n)_{bad}$. 
Also, since $\Delta_n \xrightarrow{\rm w.a.} \mu$, we have
\[
\lim_{n\to\infty} \frac{\#J(K,R,\Delta_n)}{\#V\Delta_n}=(K;\mu).
\]
Therefore
\[
\lim_{n\to\infty} \frac{(K;\Delta_n)}{\#V\Delta_n}=(K;\mu),
\] 
that is,
\[
\lim_{n\to\infty} \frac{(K;\mu_{\Delta_n})}{\#V\Delta_n}=(K;\mu).\tag{$\spadesuit$}
\]

Now let $K$ be an arbitrary finite non-degenerate subtree of $X$
containing $1\in VX$. Put $R$ to be the maximum of the distances in
$K$ from $1$ to other vertices of $K$. Then $K$ is precisely the
$R$-ball in $K$ centered at $1$, so that $K\in U_{N,R}^f$. Therefore
$(\spadesuit)$ holds for every finite subtree of $X$ containing $1$
and hence, by $F_N$-invariance of subset currents, for every
finite subtree $K$ of $X$. Hence
\[
\lim_{n\to\infty} \frac{1}{\# V\Delta_n} \mu_{\Delta_n}=\mu \text{  in
   } \gcn,
\]
as required.
\end{proof}

\begin{lem}\label{lem:blowup}
Let $\mu\in \gcn$ be a normalized current and let $\Delta_n$ be a sequence of
$R_N$-graphs such that $\Delta_n \xrightarrow{\rm w.a.} \mu$.
Then there exists a sequence of folded $R_N$-graphs $\Delta_n'$ such that $\Delta_n' \xrightarrow{\rm w.a.} \mu$.
\end{lem}
\begin{proof}

Let us say that a vertex $v$ of $\Delta_n$ is \emph{folded} if the ball of
radius $1$ in $\Delta_n$ around $v$ is a folded $R_N$-graph. 

Let $\Upsilon_N$ be the $R_N$-graph which is a simplicial circle of length $N^2$ with the
label $a_1^Na_2^N\dots a_N^N$.
For every non-folded vertex $v$ of $\Delta_n$ choose a copy
$\Upsilon_{N,v}$ of $\Upsilon_N$.
We now perform a \lq\lq blow-up\rq\rq\ operation on $\Delta_n$ as follows.  
Simultaneously, for every non-folded vertex $v$ of $\Delta_n$ we cut $\Delta_n$ open
at $v$, so that $v$ gives $deg_{\Delta_n}(v)\le 2N$ new vertices, and then
we attach these new vertices to $\Upsilon_{N,v}$ is a way that 
the resulting graph is folded.
The resulting folded $R_N$-graph is denoted by $\Delta_n'$.

The definition of weak approximation implies that $\lim_{n\to\infty}
{b_n}/{\#V\Delta_n}=0$, where $b_n$ is the number of non-folded
vertices of $\Delta_n$.
Since the number of vertices in the graph $\Upsilon_N$ (used to
blow-up non-folded vertices of $\Delta_n$) is equal to $N^2$ and does
not depend on $n$, the definition of weak approximation now implies
that $\Delta_n' \xrightarrow{\rm w.a.} \mu$, as required.

\end{proof}

\begin{lem}\label{lem:loops}
Let $\mu\in \gcn$ be a normalized current and let $\Delta_n$ be a sequence of
folded $R_N$-graphs such that $\Delta_n \xrightarrow{\rm w.a.} \mu$.
Then there exists a sequence of $R_N$-core graphs $\Delta_n'$ such that $\Delta_n' \xrightarrow{\rm w.a.} \mu$.
\end{lem}

\begin{proof}
For every vertex $v$ of degree $1$ in $\Delta_n$ attach a new
loop-edge at $v$ with label $a\in A$ such that $a^{\pm 1}$ is
different from the label of the unique edge of $\Delta_n$ starting at
$v$. Denote the resulting graph by $\Delta_n'$.

Then $\Delta_n'$ is a folded and cyclically reduced $R_N$-graph. It is easy to see,
from the definition of weak approximation, that $\Delta_n' \xrightarrow{\rm w.a.} \mu$.
\end{proof}

\subsection{Rational currents are dense}\label{subsec:result}

\begin{thm}\label{thm:dense}
The set $\mathcal SCurr_r(F_N)$ of all rational currents is a dense
subset of $\gcn$.
\end{thm}
\begin{proof}
Let $\mu\in \gcn$ be an arbitrary nonzero current. Then for some $c>0$ the current $\overline\mu:=\frac{1}{c}\mu$ is normalized.

A slight modification of the main result of~\cite{Elek}, together with
Proposition~\ref{prop:m1} imply that there exists a sequence
$\Delta_n$ of $R_N$-graphs such that $\Delta_n \xrightarrow{\rm
  w.a.} \overline\mu$.
By Lemma~\ref{lem:blowup} and Lemma~\ref{lem:loops} we may modify the sequence $\Delta_n$ to get a new sequence of $R_N$-graphs, which we again denote $\Delta_n$, such that each
$\Delta_n$ is a core graph and that $\Delta_n \xrightarrow{\rm  w.a.} \overline\mu$.
Now Lemma~\ref{lem:wa} implies that 
\[
\lim_{n\to\infty} \frac{1}{\# V\Delta_n} \mu_{\Delta_n}=\overline\mu \text{  in
   } \gcn.
\]

Let $\Delta_{n,1}, \dots \Delta_{n,k_n}$ be the connected components
of $\Delta_n$. Then 
\[
\mu_{\Delta_n}=\mu_{\Delta_{n,1}}+\dots+\mu_{\Delta_{n,k_n}}
\]
and we see that $\overline\mu$ belongs to the closure of the $\mathbb R_{\ge
  0}$-span of $\mathcal SCurr_r(F_N)$. 
Proposition~\ref{prop:span} now implies that $\overline\mu$ belongs to the
closure of $\mathcal SCurr_r(F_N)$.  Hence $\mu=c\overline \mu$ also belongs to the
closure of $\mathcal SCurr_r(F_N)$. 
Since $\mu\in \gcn$
was an arbitrary nonzero current, this completes the proof.
\end{proof}

\begin{rem}
In the proof of Theorem~\ref{thm:dense}, instead of the result of Elek~\cite{Elek} invoked above, alternatively one can use the results of Bowen~\cite{Bow03,Bow09}. 
\end{rem}

\section{The action of $\Out(F_N)$}\label{sec:action}

\subsection{Defining the $\Out(F_N)$-action}\label{subsec:defaction}

If $\phi\in \Aut(F_N)$ is an automorphism, then $\phi$ is a
quasi-isometry of $F_N$ and hence $\phi$ extends to a homeomorphism
(which we will still denote by $\phi$) $\phi:\partial F_N\to\partial
F_N$. Thus for $S\in \mathfrak C_N$ we have $\phi(S)\in \mathfrak C_N$.
For a subset $U\subseteq \mathfrak C_N$ we write $\phi(U):=\{\phi(S):
S\in U\}$. Thus $\Aut(F_N)$ has a natural action on $\mathfrak C_N$ and
on the set of subsets of $\mathfrak C_N$.

\begin{defn}[The action of $\Aut(F_N)$]\label{defn:action}
Let $\phi\in \Aut(F_N)$ and let $\mu\in \gcn$.

Define a measure $\phi\mu$ on $\mathfrak C_N$ as follows:

For a Borel subset $U\subseteq \mathfrak C_N$ we put
\[
\phi\mu(U):=\mu\left(\phi^{-1}(U)\right).
\]
\end{defn}

\begin{prop}\label{prop:action}
Let $N\ge 2$.
\begin{enumerate}
\item For any $\phi\in \Aut(F_N)$ and $\mu\in \gcn$ we have
$\phi\mu\in \gcn$.

\item For any $\phi_1,\phi_2\in \Aut(F_N)$ and $\mu\in \gcn$ we have $(\phi_1\phi_2)(\mu)=\phi_1(\phi_2\mu)$.

\item For any $\phi\in \Aut(F_N)$ $\phi:\gcn\to
\gcn$ is an $\mathbb R_{\ge 0}$-linear homeomorphism.

\item If $\phi\in \Aut(F_N)$ is an inner automorphism then $\phi\mu=\mu$
for every $\mu\in \gcn$.
\end{enumerate}
\end{prop}
\begin{proof}
To see (1) note that for any $g\in F_N$ and $\xi\in \partial F_N$ we
have $\phi^{-1}(g\xi)=\phi^{-1}(g)\phi^{-1}(\xi)$.
Hence if $U\subseteq \mathfrak C_N$ and $g\in F_N$ then, in view of
$F_N$-invariance of $\mu$, we have:
\[
(\phi\mu)(gU)=\mu(\phi^{-1}(gU))=\mu(\phi^{-1}(g) \phi^{-1}(U))=\mu(\phi^{-1}(U))=(\phi\mu)(U).
\]
Thus $\phi\mu$ is an $F_N$-invariant measure, and hence $\phi\mu\in
\gcn$ and (1) is established.

Now (2) and (3) follow directly from (1) and Definition~\ref{defn:action}.

Part (4) follows from the fact that if $\psi\in \Aut(F_N)$ is inner,
that is $\psi(g)=hgh^{-1}$ for every $g\in F_N$, then $\psi(\xi) =h\xi$
for every $\xi\in \partial F_N$. 
\end{proof}

Proposition~\ref{prop:action} shows that Definition~\ref{defn:action}
defines a left action of $\Aut(F_N)$ on $\gcn$ by
$\mathbb R_{\ge 0}$-linear homeomorphisms. Moreover, inner
automorphisms of $F_N$ lie in the kernel of this action, and therefore
this action factors through to the action of $\Out(F_N)$ on $\gcn$.

Also, by $\mathbb R_{\ge 0}$-linearity, we have
$\phi(r\mu)=r\phi(\mu)$ for $r\ge 0$, $\mu\in \gcn$ and
$\phi\in \Aut(F_N)$. Therefore the action of $\Aut(F_N)$ on $\gcn$ also yields an action of $\Aut(F_N)$ on 
$\pgcn$ given by $\phi[\mu]:=[\phi\mu]$, where $\mu\in\gcn, \mu\ne 0$ and $\phi\in \Aut(F_N)$. As before, the last
action factors through to the action of $\Out(F_N)$ on $\pgcn$.

\begin{prop}\label{prop:r-act}
Let $H\le F_N$ be a nontrivial finitely generated subgroup and let
$\phi\in \Aut(F_N)$. Then $\phi(\eta_H)=\eta_{\phi(H)}$.
\end{prop}

\begin{proof}
For $m\ge 1$, if $[H:H_1]=m$, then
$[\phi(H):\phi(H_1)]=m$ and $\eta_{H_1}=m\eta_{H}$,
$\eta_{\phi(H_1)}=m\eta_{\phi(H)}$. Hence, by linearity, it suffices to establish the
  proposition for the case $H=Comm_{F_N}(H)$.
Thus assume that $H\le F_N$ is a nontrivial finitely generated
subgroup with $H=Comm_{F_N}(H)$. Recall that
\[
\eta_{H}=\sum_{H'\sim H} \delta_{\Lambda(H')}.
\]

It is easy to see that for any subgroup $G\le F_N$ we have
$\phi(\Lambda(G))=\Lambda(\phi(G))$. 

Let $U\subseteq \mathfrak C_N$. By Definition~\ref{defn:action},
$\phi\eta_H(U)=\eta_H(\phi^{-1}(U))$ is equal to the number of elements
of the form $\Lambda(H')$ (where $H'\sim H$) that belong to
$\phi^{-1}(U)$, which, in turn, is equal to the number of elements of
the form $\phi\Lambda(H')$ (where $H'\sim H$) that belong to $U$.
Hence
\begin{gather*}
\phi(\eta_H)=\sum_{H'\sim H} \delta_{\phi\Lambda(H')}=\sum_{H'\sim H} \delta_{\Lambda(\phi(H'))}=\\
=\sum_{K\sim \phi(H)} \delta_{\Lambda(K)}=\eta_{\phi(H)},
\end{gather*}
as required.
\end{proof}

Proposition~\ref{prop:r-act} implies that the action of $\Aut(F_N)$
has a global nonzero fixed point, namely the current $\eta_{F_N}=\delta_{\partial F_N}$. Indeed, for any
$\phi\in \Aut(F_N)$ we have 
\[
\phi\eta_{F_N}=\eta_{\phi(F_N)}=\eta_{F_N}.
\]
Moreover, all scalar multiples $r\eta_{F_N}$, where $r\ge 0$, are also
fixed by $\Aut(F_N)$. In particular, since applying an automorphism to
a subgroup of $F_N$ preserves the index of the subgroup, if
$[F_N:H]<\infty$ then $\phi\eta_H=\eta_H$ for every $\phi\in \Aut(F_N)$.

\subsection{Local formulas}\label{subsec:local}

Similarly to the case of $\Curr(F_N)$, we can get more explicit \lq\lq local formulas\rq\rq , with respect to a marking, for the action of $\Out(F_N)$ on $\gcn$.

\begin{prop}\label{prop:local}
Let $\alpha:F_N\isom\pi_1(\Gamma)$ be a marking on $F_N$ and let $X$ denote the universal cover $\widetilde \Gamma$. As before (see \ref{prop:weights}), let $\mathcal K_\Gamma$ denote the set of
all non-degenerate finite simplicial subtrees of $X$. For any $\phi\in \Out(F_N)$ and any $K\in\mathcal K_\Gamma$ there exist $m\ge 1$ and $K_1\dots, K_m\in \mathcal K_\Gamma$ with the following property:
For every $\mu\in\gcn$ we have
\[
(K; \phi(\mu))_\alpha=\sum_{i=1}^m  (K_i; \mu)_\alpha.
\]
\end{prop}
\begin{proof}
Since inner automorphisms are contained in the kernel of the action of $\Aut(F_N)$ on $\gcn$, it suffices to prove the statement of the proposition under the assumption that $\phi\in \Aut(F_N)$.

Let $\phi\in \Aut(F_N)$ be arbitrary.
Since $\mathcal SCyl_\alpha(K)\subseteq \mathfrak C_N$ is compact, the set $\phi^{-1}\mathcal SCyl_\alpha(K)$ is also compact. Since subset cylinders are not just compact, but also open,  $\phi^{-1}\mathcal SCyl_\alpha(K)$ is covered by finitely many subset cylinders. By Lemma~\ref{lem:disj}, after subdividing them, we may assume that $\phi^{-1}\mathcal SCyl_\alpha(K)$ is covered by finitely many pairwise disjoint subset cylinders:
\[
\phi^{-1}\mathcal SCyl_\alpha(K) =\sqcup_{i=1}^m \mathcal SCyl_\alpha(K_i).
\]
Put $M:=\max_{i=1}^m \#EK_i$.  By definition of $\phi\mu$ we have

\begin{gather*}
(K; \phi\mu)_\alpha =\phi\mu\left(\mathcal SCyl_\alpha(K)  \right)=\mu\left(\phi^{-1}\mathcal SCyl_\alpha(K)    \right)=\\
\mu\left(\sqcup_{i=1}^m \mathcal SCyl_\alpha(K_i)  \right)=\sum_{i=1}^m \mu(\mathcal SCyl_\alpha(K_i) )=\sum_{i=1}^m ( K_i; \mu)_\alpha, 
\end{gather*}
as required.

\end{proof}

\begin{rem}
As in the case of $\Curr(F_N)$~\cite{Ka2}, a more detailed argument for the proof of Proposition~\ref{prop:local} shows that, given $\phi\in \Out(F_N)$ and $K\in \mathcal K_\Gamma$, one can algorithmically find the trees $K_1,\dots, K_m$ satisfying the conclusion of Proposition~\ref{prop:local}. 
\end{rem}

\section{The co-volume form}\label{sec:intform}

\subsection{Constructing co-volume form}\label{subsec:constr}

For $T\in\cvn$ and a nontrivial finitely generated subgroup $H\le F_N$
put
\[
||H||_T:=vol(H\setminus T_H)
\]
where $T_H$ is the smallest $H$-invariant subtree of $T$. It is also the convex hull of the limit set of $H$, see Proposition \ref{prop:basic}.

Note that if $H=\langle g\rangle$, where $g\in F_N, g\ne 1$ then
$||H||_T=||g||_T$, the translation length of $T$.

\begin{lem}\label{lem:volH}
Let $T\in\cvn$, let $H\le F_N$ be a nontrivial finitely generated
subgroup and let $\phi\in \Aut(F_N)$. Then $||\phi(H)||_{\phi T}=||H||_T$.
\end{lem}
\begin{proof}
Recall, that as a metric space, $\phi T$ is equal to $T$, but the
action of $F_N$ on $\phi T$ is defined as
\[
g\underset{\phi T\phi}{\cdot}
x=\phi^{-1}(g)\underset{T}{\cdot} x
\]
where $x\in T$, $g\in F_N$.
It follows that the tree $T_H\subseteq T$ is $\phi(H)$-invariant with respect to
the $F_N$-action on $\phi T$, and, moreover it is the smallest
$\phi(H)$-invariant subtree of $\phi T$. Also, it is easy to see that,
as metric graphs, $H\setminus T_H$ is equal to $\phi(H)\setminus (\phi
T)_{\phi(H)}$. Hence $||\phi(H)||_{\phi T}=||H||_T$, as claimed.  
\end{proof}

\begin{propdfn}[Co-volume form]\label{defn:intform}
Let $N\ge 2$. Define a map
\[
\langle\, , \, \rangle: \cvn\times\gcn\to\mathbb R_{\ge 0}
\]
as follows.
Let $T\in\cvn$ and $\mu\in \gcn$. Let
$\alpha:F_N\to\pi_1(\Gamma)$ be a marking and let $\mathcal L$ be a
metric graph structure on $\Gamma$ such that $T=(\widetilde\Gamma,
d_\mathcal L)$ in $\cvn$.
Put
\[
\langle T, \mu\rangle:=\sum_{e\in E_{top}(\Gamma)} (e;
\mu)_\alpha \mathcal L(e) \tag{$\spadesuit$}
\]
where $(e;\mu)_\alpha$ is as in Notation~\ref{not:e}.
Then the following hold:

\begin{enumerate}
\item The map $\langle\, , \, \rangle: \cvn\times\gcn\to\mathbb R_{\ge 0}$ is continuous, $\mathbb R_{\ge 0}$-linear with respect to the second argument and $\mathbb R_{\ge 0}$-homogeneous with respect to the first argument.
\item The map $\langle\, , \, \rangle$ is $\Out(F_N)$-equivariant, that is, for any $T\in\cvn$, $\mu\in \gcn$ and $\phi\in \Out(F_N)$ we have
\[
\langle \phi T, \phi\mu\rangle=\langle T,\mu\rangle .
\]
In other words, in terms of using the right $\Out(F_N)$-action on $\cvn$, for all $T\in\cvn$ and $\phi\in \Out(F_N)$ we have
\[
\langle T\phi, \mu\rangle=\langle T,\phi\mu\rangle .
\]
\item For any finitely generated nontrivial subgroup $H\le F_N$ and any $T\in \cvn$ we have
\[
\langle T, \eta_H\rangle=||H||_T
\]
where $T_H$ is the minimal $H$-invariant subtree of $T$.
\end{enumerate}
 We call the map $\langle\, , \, \rangle: \cvn\times\gcn\to\mathbb R_{\ge 0}$  \emph{the co-volume form}. 
\end{propdfn}

\begin{proof}
The continuity of $\langle\, , \, \rangle$ is a straightforward
consequence of its definition, and the argument is essentially
identical to that used in \cite{Ka2} to show continuity of the
intersection form $\cvn\times Curr(F_N)\to \mathbb R_{\ge 0}$. We
refer the reader to the proof of Proposition~5.9 in \cite{Ka2} for details.

The definition also directly implies that
\[
\langle T, c_1\mu_1+c_2\mu_2\rangle = c_1\langle T, \mu_1\rangle +c_2\langle T, \mu_2\rangle
\]
for $c_1,c_2\ge 0$, $T\in \cvn$ and $\mu_1,\mu_2\in \gcn$
and that
\[
\langle cT, \mu\rangle= c\langle T,\mu\rangle
\]
for $c\ge 0$, $T\in \cvn$, $\mu\in \gcn$. Thus (1) holds.

We now check that (3) holds. 
Let $H\le F_N$ be a nontrivial finitely generated subgroup.
Recall that,  $\Delta=H\setminus T_H$ is exactly the $\Gamma$-core graph representing the conjugacy class $[H]$ (see Convention \ref{conv:subgroupcore}. By
Theorem~\ref{thm:cc} we have $\eta_H=\mu_\Delta$.
The edges of $\Delta$ in $\Delta=H\setminus T_H$ have the same lengths as the $\mathcal L$-lengths of the corresponding edges of $\Gamma$. Thus the $\mathcal L$-volume of $\Delta$ is the sum of the $\mathcal L$-lengths of all the edges of $\Delta$, so that
\begin{gather*}
||H||_T=vol(H\setminus T_H)= vol_\mathcal L(\Delta)=\sum_{e\in E_{top}(\Gamma)} (\widetilde e; \Delta) \mathcal L(e)=\\ \sum_{e\in E_{top}(\Gamma)} (e; \mu_\Delta) \mathcal L(e)= \sum_{e\in E_{top}(\Gamma)} ( e; \eta_H) \mathcal L(e)=\langle T, \eta_H\rangle,
\end{gather*}
and (3) is verified.

We can now show that (2) holds. Since inner automorphisms of $F_N$ act
trivially on both $\cvn$ and $\gcn$, it suffices to check (2) for
elements of $\Aut(F_N)$ (rather than of $\Out(F_N)$).
Let $\phi\in \Aut(F_N)$.

We need to show that for any
 $T\in \cvn$ and $\mu\in \gcn$. 
\[
\langle \phi T, \phi\mu\rangle=\langle T,\mu\rangle.
\]
By continuity of the intersection form, already established in (1) and
by Theorem~\ref{thm:dense}, it suffices to verify the above formula
for the case where $T\in \cvn$ is arbitrary and where $\mu$ is a
rational current. Thus let $\mu=c\eta_H$ where $c\ge 0$ and $H\le F_N$
is a nontrivial finitely generated subgroup.
We have
\begin{gather*}
\langle \phi T, \phi c\eta_H\rangle=c\langle \phi T, \phi\eta_H\rangle=c\langle \phi T, \eta_{\phi(H)}\rangle=\\
c||\phi(H)||_{\phi T} =c||H||_T=c\langle T, \eta_H\rangle=\langle T, c\eta_H\rangle,
\end{gather*}
as required.
\end{proof}

\subsection{Non-existence of a continuous extension of the co-volume form to $\cvnbar$}\label{subsec:exist}

It turns out that, unlike for ordinary currents, there does not exist a continuous extension of the co-volume form to $\cvnbar\times\gcn$.
Before proving this statement, we need to recall a few background facts regarding the dynamics of the action of iwip elements and of Dehn twists elements of $\Out(F_N)$ on $\cvnbar$ and $\CVNbar$. 
Recall that an element $\phi\in \Out(F_N)$ is called an \emph{iwip} (which stands for \lq\lq irreducible with irreducible powers\rq\rq ) or \emph{fully irreducible} if there do not exist $m\ne 0$ and a proper free factor $L$ of $F_N$ such that $\phi^m([L])=[L]$, where $[L]$ is the conjugacy class of $L$. 

The following \lq\lq North-South\rq\rq\ dynamics result for iwips is well-known and was obtained by Levitt and Lustig in \cite{LL}:

\begin{prop}\cite{LL}\label{prop:LL}
Let $N\ge 2$ and let $\phi\in \Out(F_N)$ be an iwip.
Then there exist unique $[T_+]=[T_+(\phi)], [T_-]=[T_-(\phi)]\in \CVNbar$ and $\lambda_+=\lambda_+(\phi)>1$, $\lambda_-=\lambda_-(\phi)>1$ with the following properties:
\begin{enumerate}
\item We have $T_+\phi=\lambda_+ T_+$, $T_-\phi=\frac{1}{\lambda_-} T_-$, so that $[T_-]\phi=[T_-]$ and $[T_+]\phi=[T_+]$.
\item For any $[T]\in \CVNbar$ such that $[T]\ne [T_-]$ we have $\lim_{n\to\infty} [T]\phi^n = [T_+]$ and for any $[T]\in \CVNbar$ such that $[T]\ne [T_+]$ we have $\lim_{n\to\infty} [T]\phi^{-n} = [T_-]$.
\item For any $T\in \cvn$ we have $\lim_{n\to\infty} [T]\phi^n =[T_+]$. Moreover, given any $T\in \cvn$, we can choose a representative $T_+\in\cvnbar$ of $[T_+]$ such that
\[
\lim_{n\to\infty} \frac{1}{\lambda_+^n} T\phi^n = T_+  \quad \text{ in } \quad \cvnbar.
\]
\item The action of $F_N$ on $T_+$ has dense $F_N$-orbits.  Moreover, if $N\ge 3$ and $\phi\in \Out(F_N)$ is an atoroidal iwip then the action of $F_N$ on $T_+$ is also free.
\item If $\theta\in \Out(F_N)$ is arbitrary then for $\psi= \theta^{-1} \phi\theta$ we have $\lambda_\pm(\psi)=\lambda_{\pm}(\phi)$ and $[T_\pm (\psi)]=[T_\pm(\phi)]\theta$.  Moreover, if $T\in\cvn$ and $T':=\lim_{n\to\infty} \frac{1}{\lambda_+^n} T\phi^n$ and $T'':=\lim_{n\to\infty} \frac{1}{\lambda_+^n} T\psi^n$ then $T''=T'\theta$. 

 \end{enumerate}
The trees $[T_+(\phi)]$ and $[T_-(\phi)]$ are called the \emph{attracting} and \emph{repelling} trees of $\phi$ respectively.
\end{prop}

The attracting tree $T_+(\phi)$ of an iwip $\phi$ can be understood fairly explicitly in terms of a train-track representative of $\phi$ and in fact $\lambda_+(\phi)$ is the Perron-Frobenius eigenvalue of any train-track representative of $\phi$.

Guirardel proved in \cite{Gui1} that for $N\ge 3$ the action of $\Out(F_N)$ on $\partial \CVN=\CVNbar - \CVN$ is not topologically minimal and that there exists a unique proper closed $\Out(F_N)$-invariant subset of $\partial \CVN$. We only need the following limited version of his result:

\begin{prop}\label{prop:Gui}\cite{Gui1}
Let $N\ge 3$. Then there exists a unique minimal nonempty closed $\Out(F_N)$-invariant subset $\mathcal M_N^{cv}\subseteq \partial \CVN$ (so that for every $[T]\in \mathcal M_N^{cv}$ the $\Out(F_N)$-orbit of $[T]$ is dense in $\mathcal M_N^{cv}$). Moreover, if $T_\ast$ is the Bass-Serre tree (with all edges given length $1$) of any nontrivial free product decomposition $F_N=B\ast C$, then $[T_\ast]\in \mathcal M_N^{cv}$. 
\end{prop}

Proposition~\ref{prop:LL} easily implies that for any iwip $\phi\in \Out(F_N)$ we have $[T_+(\phi)]\in \mathcal M_N^{cv}$.

\begin{prop}\label{prop:disc}
Let $N\ge 3$. Let $F_N=B\ast C$ be a nontrivial free product decomposition and let $T_\ast\in\cvnbar$ be the corresponding Bass-Serre tree. Then there exist sequences $T_n,T_n'\in \cvn$ such that $\lim_{n\to\infty} T_n=\lim_{n\to\infty} T_n'=T_\ast$ in $\cvnbar$ and that
\[
\lim_{n\to\infty} \langle T_n', \eta_{F_N}\rangle =1
\]
but
\[
\lim_{n\to\infty} \langle T_n, \eta_{F_N}\rangle =0.
\]
\end{prop}
\begin{proof}
Let $\phi\in \Out(F_N)$ be any iwip. Choose an arbitrary point $T\in \cvn$. 
Let $[T_+]=[T_+(\phi)]$ be the attracting tree of $\phi$. We may assume that $T_+(\phi)=\lim_{k\to\infty}\frac{1}{\lambda_+^k} T\phi^k$. 
Since both $[T_\ast]$ and $[T_+]$ belong to $\mathcal M_N^{cv}$, there exists a sequence $\theta_n\in \Out(F_N)$ such that $\lim_{n\to\infty} [T_+]\theta_n= [T_\ast]$. Thus for some sequence $c_n>0$ we have $\lim_{n\to\infty} c_n T_+\theta_n= T_\ast$ in $\cvnbar$. By Proposition~\ref{prop:LL} we know that $[T_+]\theta_n=[T_+(\psi_n)]$ where $\psi_n=\theta_n^{-1} \phi \theta_n$.  Moreover, for each $n\ge 1$ we have
\[
T_+\theta_n=\lim_{k\to\infty} \frac{1}{\lambda_+^k} T\psi_n^k
\]
Put $T_+(\psi_n):=T_+\theta_n$, and we then have $T_+\theta_n=T_+(\psi_n)= \lim_{k\to\infty} \frac{1}{\lambda_+^k} T \psi_n^k$ in $\cvnbar$.
For every $k\ge 1$ and $n\ge 1$ we have 
\[
\langle \frac{1}{\lambda_+^k} T \psi_n^k, \eta_{F_N}\rangle=\frac{1}{\lambda_+^k}\langle T, \psi_n^k\eta_{F_N}\rangle=\frac{1}{\lambda_+^k}\langle T, \eta_{F_N}\rangle=\frac{\vol(F_N\setminus T)}{\lambda_+^k}.
\]
For each $n\ge 1$ choose $k_n\ge 1$ such that for all $k\ge k_n$ we have $c_n\langle \frac{1}{\lambda_+^k} T \psi_n^k, \eta_{F_N}\rangle=c_n \frac{\vol(F_N\setminus T)}{\lambda_+^{k}}\le \frac{1}{n}$.

Since $\cvnbar$ is metrizable, and since $\lim_{n\to\infty} c_n T_+(\psi_n)= T_\ast$ and $c_nT_+(\psi_n)= \lim_{k\to\infty} \frac{c_n}{\lambda_+^k} T \psi_n^k$,
by a standard diagonalization argument we can find a sequence $m_n\ge k_n$ such that $\lim_{n\to\infty} \frac{c_n}{\lambda_+^{m_n}} T \psi_n^{m_n} =T_\ast$ in $\cvnbar$.  Put $T_n=\frac{c_n}{\lambda_+^{m_n}} T \psi_n^{m_n}$. Then by construction we have $\lim_{n\to\infty} T_n=T_\ast$ and $\langle T_n, \eta_{F_N}\rangle \le \frac{1}{n}\to 0$ as $n\to\infty$, so that
\[
\lim_{n\to\infty} \langle T_n, \eta_{F_N}\rangle =0.
\]

To construct the sequence $T_n'$, choose a free basis $\{b_1,\dots,b_i\}$ of $B$ and a free basis $\{c_{i+1},\dots, c_N\}$ of $C$. Let $\Gamma$ be the \lq\lq barbell\rq\rq\ graph (with the obvious marking) consisting of a non-loop edge $e$ with $i$ loop-edges (corresponding to $b_1,\dots, b_i$) attached at $o(e)$ and with $N-i$ loop-edges (corresponding to $c_{i+1},\dots, c_N$) attached at $t(e)$. Give the edge $e$ length $1-\frac{1}{n}$ and give each of $N$ loop-edges in $\Gamma$ length $\frac{1}{Nn}$. This defines a point $T_n'\in\cvnbar$ with $\vol(F_N\setminus T_n')=1$. Also, by construction, $\lim_{n\to\infty} T_n'=T_\ast$ in $\cvnbar$. Then 
\[
\lim_{n\to\infty} \langle T_n', \eta_{F_N}\rangle =\lim_{n\to\infty} \vol(F_N\setminus T_n')=1,
\]
as required.
\end{proof}

Proposition~\ref{prop:disc} immediately implies:

\begin{thm}\label{thm:disc}
Let $N\ge 3$. Then the co-volume form $\cvn\times \gcn\to R_{\ge 0}$ does not admit a continuous extension to a map $\cvnbar\times \gcn\to R_{\ge 0}$.
\end{thm}

The proof of Theorem~\ref{thm:disc} can be modified to cover the case $N=2$, but we only deal with the case $N\ge 3$ for simplicity.

\begin{rem}\label{rem:H}
Let $H\le F_N$ be a nontrivial finitely generated subgroup. 
Recall that if $T\in\cvn$ then $\langle T, \eta_H\rangle=||H||_T=\vol(H\setminus T_H)$.  If $T\in\cvnbar-\cvn$, the quotient $H\setminus T_H$ is, in general, not a nice object, and, in particular, it is not necessarily a finite metric graph.  However, one can still define a reasonable notion of \lq\lq volume\rq\rq\  $||H||_T$ for $H\setminus T_H$. 
Namely, if $H$ fixes a point of $T$, put $||H||_T:=0$. Otherwise, there exists a unique minimal $H$-invariant subtree $T_H$ of $T$. In that case define $||H||_T$ as the infimum of $\vol(K)$ taken over all finite subtrees $K\subseteq T_H$ such that $HK=T_H$.  The proof of Proposition~\ref{prop:disc} exploits the fact that in general, given $H$, the function $f_H:\cvnbar\to\mathbb R$, $f_H:T\mapsto ||H||_T$, is not continuous on $\cvnbar$. However, one can show that $f_H$ is upper-semicontinuous.

\end{rem}

\section{The reduced rank functional}\label{sec:rrf}

Recall that for a finitely generated free group $F$ the \emph{rank}
$\rk(F)$ is the cardinality of a free basis of $F$ and the
\emph{reduced rank} $\rrk(F)$ is defined as
\[
\rrk(F):=\max\{\rk(F)-1, 0\}.
\] 
Thus is $F\ne \{1\}$ then $\rrk(F)=\rk(F)-1$. If $\Delta$ is a finite
connected graph and $F=\pi_1(\Delta)$ then $\rrk(F)=-\chi(\Delta)$,
where $\chi(\Delta)$ is the Euler characteristic of $\Delta$.
Reduced rank appears naturally in the context of the Hanna Neumann
Conjecture, recently proved by Mineyev~\cite{Min}. The conjecture (now Mineyev's theorem) states that if $H_1,H_2\le F$ are finitely
generated subgroups of a free group $F$ then
\[
\rrk(H_1\cap H_2)\le \rrk(H_1)\rrk(H_2).
\]
It turns out that reduced rank extends to a $\mathbb R_{\ge 0}$-linear
functional on the space of subset currents:

\begin{thm}\label{thm:rrk}
Let $N\ge 2$. Then there exists a unique continuous 
$\mathbb R_{\ge 0}$-linear
functional
\[
\rrk: \gcn\to\mathbb R_{\ge 0}
\]
such that for every nontrivial finitely generated subgroup $H\le F_N$
we have
\[
\rrk(\eta_H)=\rrk(H).\tag{$\diamondsuit$}
\]
Moreover, $\rrk$ is $\Out(F_N)$-invariant, that is, for any
$\phi\in\Out(F_N)$ and $\mu\in \gcn$ we have $\rrk(\phi\mu)=\rrk(\mu)$.
\end{thm}

\begin{proof}
The uniqueness of $\rrk$ follows from $(\diamondsuit)$ and from the requirement that $\rrk$ be linear
and continuous, since this uniquely defines $\rrk$ on the set of
rational subset currents, which is dense in $\gcn$ by
Theorem~\ref{thm:dense}.
Thus it suffices to prove the existence of a functional $\rrk$ with
the required properties.

Choose a free basis $A$ of $F_N$ and the corresponding marking
$\alpha_A:F_N\to\pi_1(R_N)$ as in Convention~\ref{conv:GA}. Recall that $X=\widetilde R_N$ is the Cayley graph of $F_N$ with respect to $A$.
For each $a\in A$ let $e_a$ be the topological edge in $X$ with
endpoints $1,a\in F_N$. Note that every subgraph in $X$ consisting of
a single edge is a translate of some $e_a$ by an element of $F_N$.

Let $\mathcal B_N$ be the set of all non-degenerate finite subtrees $K$
contained in the ball of radius $1$ in $X$ with center $1\in F_N$ such
that the vertex $1\in F_N$ has degree $\ge 2$ in $K$. Note that every
$K\in \mathcal B_N$ is uniquely specified by the set of its 
vertices of degree 1, which is a subset of $A\cup A^{-1}$ of cardinality at least
$2$. Thus there are exactly $2^{2N}-2N-1$ elements in $\mathcal B_N$.
Note also that if $\Delta$ is a nontrivial finite $R_N$-core graph then for every
vertex $x$ of $\Delta$ there exists a unique $K\in \mathcal B_N$ such
that $Lk_\Delta(x)=Lk_{K}(1)$.
Define the function $\rrk:\gcn\to\mathbb R$ as follows. For $\mu\in\gcn$
\[
\rrk(\mu):=\sum_{a\in A} (e_a; \mu) - \sum_{K\in \mathcal
  B_N} (K;\mu).
\]
By construction $\rrk:\gcn\to\mathbb R$ is a continuous $\mathbb R_{\ge 0}$-linear function.

Suppose now that $H\le F_N$ is a nontrivial finitely generated
subgroup. Let $\Delta$ be the  finite $R_N$-core graph representing $[H]$.
Then
\begin{gather*}
\rrk(\mu_\Delta):=\sum_{a\in A} (e_a; \mu_\Delta) - \sum_{K\in \mathcal
  B_N} (K;\mu_\Delta)=\\
\sum_{a\in A} (e_a; \Delta) - \sum_{K\in \mathcal
  B_N} (K;\Delta).
\end{gather*}
It is easy to see, from the definition of an occurrence (Definition \ref{defn:occur}), that
$\sum_{a\in A} (e_a; \Delta)=\#E_{top}\Delta$ and that $\sum_{K\in \mathcal B_N} (K;\Delta)=\#V\Delta$.
Thus
\[
\rrk(\eta_H)=\rrk(\mu_\Delta)=\#E_{top}\Delta-\#V\Delta=-\chi(\Delta)=\rrk(H).
\]
Note that, by linearity, for any $c\ge 0$ $\rrk(c\eta_H)=c\rrk(H)\ge
0$. Thus $\rrk\ge 0$ on a dense subset of $\gcn$ and hence, by
continuity, $\rrk(\mu)\ge 0$ for every $\mu\in \gcn$.

Suppose now that $\phi\in \Aut(F_N)$. We claim that
$\rrk(\phi\mu)=\rrk(\mu)$ for every $\mu\in\gcn$. Indeed, for any finitely generated subgroup $H\le F_N$, $H$ is isomorphic to $\phi(H)$ and hence $\rrk(H)=\rrk(\phi(H))$. Thus, in view of $(\diamondsuit)$, the claim holds for every rational
subset current and hence, since by Theorem~\ref{thm:dense} rational currents are dense in $\gcn$, the claim holds for every
$\mu\in \gcn$. This shows that $\rrk$ is $\Aut(F_N)$-invariant, and hence $\Out(F_N)$-invariant as well.
\end{proof}

\section{Analogs of uniform currents}\label{sec:uniform}

In~\cite{Ka2}, given a free basis $A$ of $F_N$ we constructed a
\emph{uniform current} $m_A\in\Curr(F_N)$ associated to $A$. Intuitively,
the current $m_A$ \lq\lq splits equally\rq\rq\ in all directions in the Cayley
graph $X$ of $F_N$ with respect to $A$.

In the setup of subset currents, given a free basis $A$ of $F_N$
one can define a natural family $m_{A,d}\in \gcn$, $d=2,\dots, 2N$ of
\lq\lq uniform subset currents\rq\rq , with $m_{A,2}=m_A$ and
$m_{A,2N}=\eta_{F_N}$. The current $m_{A,d}$ is \lq\lq supported\rq\rq\ on
$d$-regular subtrees of $X$. That is, $m_{A,d}$ will have the property
that $(K;m_{A,d}) >0$ if and only if $K\subseteq X$ is a
finite non-degenerate subtree where every vertex has degree either $1$ or  $d$ in $K$.

Before giving the explicit definition of $m_{A,d}$ we present the
following computation as motivation. Since $m_{A,d}$ is a subset current, we need it to satisfy the weights condition $(\bigstar)$ from Proposition~\ref{prop:kirch}.
Let $K$ be a non-degenerate finite subtree of $X$ where every vertex has degree $1$ or $d$ in $K$. Let $e$ be a terminal edge of $K$, as in Definition \ref{defn:gcyl}, and let $a\in A^{\pm 1}$ be the label of $e$. Then, in notations \ref{notation:edges}, $q(e)$ consists of precisely $2N-1$ distinct edges, with
labels from $A^{\pm 1}-\{a^{-1}\}$. We want to choose a nonempty subset $U\subseteq q(e)$ so that in the tree $K'=K\cup U$ the terminus of $e$ has degree $d$. Thus the set $U$ needs to have cardinality $d-1$. Hence there are exactly $\binom{2N-1}{d-1}$ choices
for $U\subseteq q(e)$. Since $m_{A,d}$ is supposed to be the \lq\lq most symmetric possible\rq\rq , we assign each of these choices equal weight, so we want
\[
(K\cup U; m_{A,d})=\frac{(K;m_{A,d})}{\binom{2N-1}{d-1}}
\] 
for every subset $U\subseteq q(e)$ of size $d-1$. 
For nonempty subsets $U\subseteq q(e)$ of size different from $d-1$ we want $(K\cup U; m_{A,d})=0$.
Then, using notations \ref{notation:edges}, we will have
\[
( K; m_{A,d})=\sum_{U\in P_+(q(e))} (K\cup U; m_{A,d})
\]
as required by $(\bigstar)$. Assigning every one-edge subtree of $X$ weight $1/N$ in $m_{A,d}$ and iterating the above splitting formula
yields the following:

\begin{propdfn}[Uniform subset currents]
Let $2\le d\le 2N$. Then there exists a unique subset current $m_{A,d}\in \gcn$ such that
\begin{enumerate}
\item if $K\subseteq X$ is a finite subtree with $n\ge 1$ edges and
  with every  vertex of degree either $1$ or $d$ in $K$, then
\[
( K; m_{A,d})=\frac{1}{N\left(\binom{2N-1}{d-1}\right)^{n-1}} ;
\]
\item if $K\subseteq X$ is a finite subtree where some
  vertex has degree different from $d$ and from $1$, then 
\[
( K; m_{A,d})=0 ;
\]
\item We have $\langle X,  m_{A,d}\rangle=1$.
\end{enumerate}
The current $m_{A,d}\in \gcn$ is called the \emph{uniform subset current} of grade $d$ on $F_N$ corresponding to $A$. 
\end{propdfn}

\begin{proof}
It is easy to check, via a direct computation, that the weights
$(K; m_{A,d})$ specified above satisfy condition
$(\bigstar)$ and hence, by Proposition~\ref{prop:kirch}, there does
exist a unique subset current realizing these weights.
Also, for every single-edge tree $K$ we have $(K; m_{A,d})=\frac{1}{N}$, and this normalization ensures 
that  $\langle X,  m_{A,d}\rangle=1$.
For $d=2N$ the above formulas give $(K; m_{A,2N})=1$ if
$K\subseteq X$ is a finite $2N$-regular subtree and $(K;
m_{A,2N})=0$ if $K$ is not $2N$-regular. This shows that $m_{A,2N}=\eta_{F_N}$.
Also, for $d=2$, the definition of $m_{A,2}$ yields 
$(K; m_{A,2})=\frac{1}{2N(2N-1)^n}$ if $K\subseteq X$ is a
linear segment of length $n\ge 1$ and $(K; m_{A,2})=0$ if
$K$ is not 2-regular. Thus $m_{A,2}=m_A\in\Curr(F_N)$, as claimed.
\end{proof}

In a similar way, uniform subset currents can be defined with respect to any
marking $\alpha:F_N\to\pi_1(\Gamma)$, where $\Gamma$ is a $k$-regular
graph (and not necessarily the standard rose).

Perhaps a more interesting notion is that of the \emph{absolute uniform current} $m^\mathcal S_{A}\in \gcn$ associated to $A$.

If $K$ is a finite non-degenerate tree, we say that a vertex $x$ of $K$ is \emph{interior} if $x$ has degree $\ge 2$ in $K$. Denote by $\iota(K)$ the number of interior vertices in $K$.  Before giving a formal definition of $m^\mathcal S_{A}$ let us again start with some motivation. We want $m^\mathcal S_{A}$ to have the property that for a finite non-degenerate subtree $K$ of $X$ the weight $( K; m^\mathcal S_{A})$ depends only on $\iota(K)$.

We build finite subtrees of $X$ step-by-step, starting with a tree $K_0$ consisting of a single edge. Since  there are exactly $N$ distinct $F_N$-translation classes of topological edges in $X$, we assign every one-edge subtree weight $\frac{1}{N}$ in $m^\mathcal S_{A}$.  Note that $\iota(K_0)=0$. Arguing inductively, let $n\ge 0$ and suppose $K_n$ is already constructed and that $\iota(K_n)=n$. We choose $e$ to be any (oriented) leaf  of $K$. Then $q(e)$ (see notation \ref{notation:edges}) consists of precisely $2N-1$ distinct edges, with
labels from $A^{\pm 1}-\{a^{-1}\}$. The set $P_+(q(e))$ of nonempty subsets of $q(e)$ has exactly $2^{2N-1}-1$ elements. We choose any nonempty subset $U$ of $q(e)$ \lq\lq with equal probability\rq\rq , and put $K_{n+1}=K\cup U$. Note that the terminus of $e$ has become an interior vertex of $K_{n+1}$, so that $\iota(K_{n+1})=n+1$.
Since we are supposed to choose $U\in P_+(q(e))$ uniformly at random, we want 
\[
(K\cup U; m_{A,d})=\frac{(K; m_{A,d})}{2^{2N-1}-1}
\] 
for every nonempty subset $U\subseteq q(e)$.
This choice will assure that
\[
( K;m_{A,d})=\sum_{U\in P_+(q(e))} (K\cup U; m_{A,d})
\]
as required by condition $(\bigstar)$ of Proposition~\ref{prop:kirch}. 
The above considerations lead to the following:

\begin{propdfn}[Absolute uniform current]\label{propdfn:mag}
Let $N\ge 2$, $A$ be a free basis of $F_N$ and let $X$ be the Cayley graph of $F_N$ with respect to $A$.

Then there exists a unique subset current
$m_{A}^\mathcal S\in \gcn$ such that:
\begin{enumerate}
\item If $K\subseteq X$ is a finite non-degenerate subtree with $\iota(K)=n\ge 0$ then
\[
(K; m_{A}^\mathcal S)=\frac{1}{N\left( 2^{2N-1}-1  \right)^{n-1}}.
\]
\item We have $\langle X,  m_{A,d}\rangle=1$.
\end{enumerate}

The current $m_{A}^\mathcal S\in \gcn$ is called the \emph{absolute uniform subset
  current} corresponding to $A$. 
\end{propdfn}
\begin{proof}
It is not hard to check that the weights given in the definition  of $m_A^\mathcal S$  satisfy the \lq\lq switch\rq\rq\  condition $(\bigstar)$ of Proposition~\ref{prop:kirch}, which directly yields the above statement.
\end{proof}

Note that, unlike the currents $m_{A,d}$ constructed earlier, the current $m_A^\mathcal S$ has \lq\lq full support\rq\rq , meaning that $(K;m_A^\mathcal S) >0$ for \emph{every} finite non-degenerate subtree $K$ of $X$. Recall that, by Proposition~\ref{prop:m1}, there is a canonical $\mathbb R_{\ge 0}$-linear homeomorphism between $\gcn$ and the space $\mathcal M_1(X)$ of finite positive Borel measures on the space $\mathcal T_1(X)$ of rooted subtrees of $X$ with root-vertex $1$, which are invariant with respect to root-change. Under this homeomorphism $m_A^\mathcal S$ corresponds to a probability measure $\mathbf m_A^\mathcal S$ on $\mathcal T_1(X)$ and $\mathbf m_A^\mathcal S$-random points of  $\mathcal T_1(X)$ appear to represent an interesting class of subtrees of $X$.

\section{Open problems}\label{sec:problems}

\subsection{Subset currents on surface groups}\label{subsec:surface}

As noted in the introduction, the notion of a subset current makes sense for any non-elementary word-hyperbolic group $G$. Apart from $F_N$, a particularly interesting case is that of a surface group. Namely, let $\Sigma$ be a closed oriented surface of negative Euler characteristic and let $G=\pi_1(\Sigma)$.  Put a hyperbolic Riemannian metric $\rho$ on $\Sigma$ so that $\widetilde \Sigma$ with the lifted metric becomes isometric to the hyperbolic plane $\mathbb H^2$.  The action of $G$ on $\widetilde \Sigma=\mathbb H^2$ by covering transformations is a free isometric discrete co-compact action, so that $G$ is quasi-isometric to $\mathbb H^2$ and $\partial G$ is $G$-equivariantly homeomorphic to $\partial \mathbb H^2=\mathbb S^1$.  As for a free group we can consider the space $\mathfrak(\mathbb S^1)$ of all closed subsets $S\subseteq \mathbb S^1$ such that $\#S\ge 2$. A \emph{subset current} on $G$ is a locally finite $G$-invariant positive Borel measure on $\mathfrak(\mathbb S^1)$. The space $\mathcal S \Curr(G)$ of all subset currents on $G$ again comes equipped with a natural weak-* topology and a natural action of the mapping class group $Mod(\Sigma)$.  In this context we can again define the notion of a \emph{counting current} associated with a nontrivial finitely generated subgroup $H\le G$. If $H=Comm_G(H)$ then we put
\[
\eta_H:=\sum_{H_1\in [H] }\delta_{\Lambda(H_1)}
\] 
where $[H]$ is the conjugacy class of $H$ in $G$ and where for a subgroup $H_1\in [H]$ $\Lambda(H_1)\subseteq \mathbb S^1$ is the limit set of $H_1$ in $\mathbb S^1$. 
If $H\le G$ is an arbitrary nontrivial finitely generated subgroup, then, by well-known results, $H$ has a finite index $m\ge 1$ in its commensurator $H_0:=Comm_G(H)$ and we put $\eta_H:=m\  \eta_{H_0}$.
\begin{prob}\label{prob:surface}
In the above set-up, is it true that the set
\[
\{ c\eta_H | c\ge 0, H\le G \text{ is a nontrivial finitely generated subgroup} \}
\]
is dense in $\mathcal S \Curr(G)$?
\end{prob}
Note that for every $S\in \mathfrak(\mathbb S^1)$ one can consider the convex hull $Conv(S)\subseteq \mathbb H^2$, and therefore one can geometrically view a subset current on $G$ as a $G$-invariant measure on the space of \lq\lq nice\rq\rq\  convex subsets of $\mathbb H^2$. If $H\le G$ is a nontrivial finitely generated subgroup of infinite index, then $H$ is a free group of finite rank $N\ge 1$ and $Conv(\Lambda H)\subseteq \mathbb H^2$ is an $H$-invariant subset which is quasi-isometric to $F_N$ (i.e. it is a \lq\lq quasi-tree\rq\rq ). However, the machinery of measures on rooted graphs that we used to show that rational currents are dense in  $\gcn$ is not directly applicable for tackling Problem~\ref{prob:surface}.

\subsection{Continuity of co-volume for a fixed tree $T\in\cvnbar$}\label{subsec:continuity}

In Remark~\ref{rem:H} we defined the notion of co-volume $||H||_T$ where $H\le F_N$ is any nontrivial finitely generated subgroup and where $T\in \cvnbar$. We have seen that for a fixed $H$, the function $\cvnbar\to [0,\infty)$, $T\to ||H||_T$, is not necessarily continuous on $\cvnbar$ (although it is, of course, continuous if $H$ is infinite cyclic).

One can still ask if, given a fixed $T\in\cvnbar$, the co-volume $||.||_T$ extends to a continuous function on $\gcn$:

\begin{prob}\label{prob:covol}
Let $T\in\cvnbar$. Suppose $\mu\in \gcn$ and that $c_n\eta_{H_n}\in \gcn$ are rational currents such that $\lim_{n\to\infty} c_n\eta_{H_n}=\mu$. Does this imply that
\[
\lim_{n\to\infty} c_n||H_n||_{T}
\]
exists and is independent of the sequence  $c_n\eta_{H_n}$ approximating $\mu$? If yes, we will denote the above limit by $||\mu||_T$. 
\end{prob}

\subsection{Volume equivalence}\label{subsec:equivalence}

Kapovich, Levitt, Schupp and Shpilrain~\cite{KLSS} introduced and studied the notion of \emph{translation equivalence} in free groups. Namely, two elements $g,h\in F_N$ are \emph{translation equivalent} in $F_N$, denoted $g\equiv_t h$,  if for every $T\in\cvn$ we have $||g||_T=||h||_T$. It is easy to see that $g\equiv_t h$ in $F_N$ if and only if for every $T\in\cvnbar$  $||g||_T=||h||_T$. Similarly, we say (see \cite{Ka2}) that two currents $\mu_1,\mu_2\in\Curr(F_N)$ are \emph{translation equivalent} if for every $T\in \cvn$ $\langle T, \mu_1\rangle=\langle T,\mu_2\rangle$.  Again, it clear that replacing $\cvn$ by $\cvnbar$ in this definition yields the same notion.  Several different sources of translation equivalence (in particular traces of $SL(2,\mathbb C)$-representations of free groups), have been exhibited in~\cite{KLSS}, and further results were obtained in~\cite{Lee,Lee2,LV}. The paper~\cite{KLSS} also defined the notion of \emph{volume equivalence} in free groups. Two nontrivial finitely generated subgroups $H_1,H_2\le F_N$ are said to be \emph{volume equivalent} in $F_N$, denoted $H_1\equiv_v H_2$,  if for every $T\in \cvn$ we have $||H_1||_T=||H_2||_T$. Note that for nontrivial $g,h\in F_N$ we have $\langle g\rangle\equiv_v \langle h\rangle$ in $F_N$ if and only if $g\equiv_t h$ in $F_N$.  It is clear that conjugate subgroups are volume equivalent and that if $H_1,H_2\le F_N$ are nontrivial finitely generated subgroups such that $Comm_{F_N}(H_1)=Comm_{F_N}(H_2)$ and such that $H_1$ and $H_2$ have the same index in $Comm_{F_N}(H_1)$, then $H_1\equiv_v H_2$ in $F_N$. Similarly, the definition of volume equivalence easily implies that if $H_1\equiv_v H_2$ in $F_N$ and $\psi: F_N\to F_M$ is an injective homomorphism then $\psi(H_1)\equiv_v \psi(H_2)$ in $F_M$. Some more interesting sources of volume equivalence were found by Lee and Ventura in~\cite{LV}, who also exhibited an example of a cyclic subgroup that is volume equivalent to a subgroup that is free of rank 2.  The definition of volume equivalence naturally leads to the following question:

\begin{prob}\label{prob:ve}
Suppose $H_1\equiv_v H_2$ in $F_N$. Does this imply that for every $T\in\cvnbar$ $||H_1||_T=||H_2||_T$?
\end{prob}
By analogy with the translation equivalence case, we can also say that two subset currents $\mu_1,\mu_2\in \gcn$ are \emph{volume equivalent} in $F_N$, denoted $\mu_1\equiv_v \mu_2$, if for every $T\in \cvn$ $\langle T, \mu_1\rangle=\langle T,\mu_2\rangle$.  Thus for finitely generated subgroups $H_1,H_2\le F_N$ we have $H_1\equiv_v H_2$ if and only if $\eta_{H_1}\equiv_v \eta_{H_2}$. The notion of volume equivalence for subset currents measures the degeneracy of the co-volume form $\langle\, ,\, \rangle$ with respect to its second argument, and one can also pose an analog of Problem~\ref{prob:ve} for subset currents.

\subsection{Generalizing the Stallings fiber product construction}\label{subsec:fiber}

Let $H, L\le F_N$ be nontrivial finitely generated subgroups and let $\alpha:F_N\isom\pi_1(\Gamma)$ be a marking on $F_N$.
Let $\Delta_H, \Delta_L$ be the  finite connected $\Gamma$-core graphs representing $[H]$ and $[L]$ respectively.
Then one can define (see~\cite{Sta,KM}) the
\lq\lq fiber product graph\rq\rq\ $\Delta_H\times \Delta_L$. This graph is again a finite folded $\Gamma$-graph but not necessarily connected. The
fundamental groups of the non-contractible connected components of
$\Delta_H\times \Delta_L$ (if there are any such
components) represent all the possible $F_N$-conjugacy classes of nontrivial
intersections of the form $gHg^{-1}\cap L$, where $g\in F_N$. There are finitely many such non-contractible components and
  denote the conjugacy classes of subgroups of $F_N$ represented by
  them by $[U_1],\dots, [U_k]$. We thus define
\[
\pitchfork(\eta_H,\eta_L)=\sum_{i=1}^k \eta_{U_i}.
\]
Note that $\rrk(\pitchfork(\eta_H,\eta_L))=\sum_{i=1}^k \rrk(\eta_{U_i})=\sum_{i=1}^k \rrk(U_i)$.
If all connected components of $\Delta_H\times \Delta_L$ are contractible, define $\pitchfork(\eta_H,\eta_L)=0$. We can extend $\pitchfork$ by homogeneity to the set of rational
subset currents as
\[
\pitchfork(c_1\eta_H,c_2\eta_L):=c_1c_2\pitchfork(\eta_H,\eta_L).
\]
Note that
$\pitchfork(c_1\eta_H,c_2\eta_L)=\pitchfork(c_2\eta_L,c_1\eta_H)$ and
that $\pitchfork(c_1\eta_H,c_2\eta_L)=0$ if at least one of $H,L$ is infinite
cyclic. A particularly intriguing question is the following:
\begin{prob}\label{prob:pitch}
Does the map $\pitchfork$ extends to a continuous function
\[
\pitchfork:\gcn\times\gcn\to\gcn?
\]
\end{prob}
If yes, then composing  $\pitchfork$ with the reduced rank functional
$\overline{\rm rk}$ would give us a continuous, bilinear and \emph{symmetric}
\lq\lq intersection functional\rq\rq\ $J:= \overline{\rm rk}\, \circ \pitchfork$
\[
J : \gcn\times\gcn\to \mathbb R.
\]
Moreover, in view of Mineyev's recent proof of the Strengthened Hanna
Neumann Conjecture~\cite{Min}, it would follow that for any
$\mu_1,\mu_2\in \gcn$ we have
\[
J(\mu_1,\mu_2)\le \overline{\rm rk}(\mu_1) \overline{\rm rk}(\mu_2).
\]

\subsection{Random subgroup graphs and uniform currents}\label{subsec:random}

Let $A$ be a free basis of $F_N$.

Let $\xi=x_1x_2\dots x_n\dots \in \partial F_N$
be a \lq\lq random\rq\rq\ geodesic ray over $A^{\pm 1}$, that is, $\xi$ is a random trajectory of the non-backtracking simple random walk on $F_N$ corresponding to $A$. Denote $\xi|_n=x_1\dots x_n\in F_N$. 
It is shown in~\cite{Ka2} for a.e. $\xi\in \partial F_N$ we have
\[
\lim_{n\to\infty} \frac{\eta_{\xi|_n}}{n}=m_A \text{ in } \Curr(F_N)
\]
where $m_A\in \Curr(F_N)$ is the \emph{uniform current} on $F_N$ corresponding to $A$.
For a random $\xi\in \partial F_N$ we may think of $\langle \xi|_n\rangle\le F_N$ as a \lq\lq random\rq\rq\ cyclic subgroup of $F_N$.

In Section~\ref{sec:uniform} we have defined a family of uniform subset currents $m_{A,d}\in \gcn$, for $d=2,\dots, 2N$, where $m_{A,2}=m_A\in \Curr(F_N)$.
Recall that for a finite non-degenerate subtree $K$ of the Cayley graph $X$ of $F_N$ with respect to $A$ we have $( K; m_{A,d}) >0$ if and only if $K$ is a $d$-regular tree.

\begin{prob}  Let $3\le d\le 2N-1$.
Define a reasonable discrete-time random process such that at time $n$ it outputs a $d$-regular finite $R_N$-core graph $\Delta_n$ with $\lim_{n\to\infty} \#V\Delta_n=\infty$ and show that for a.e. trajectory of this process we have
\[
\lim_{n\to\infty}\frac{\eta_{\Delta_n}}{\#V\Delta_n}=m_{A,d} \text{ in }\gcn.
\]
\end{prob}
The same problem is also interesting for the absolute uniform current $m_A^\mathcal S\in \gcn$.

\begin{prob}
Define a reasonable discrete-time random process such that at time $n$ it outputs a finite $R_N$-core graph $\Delta_n$ with $\lim_{n\to\infty} \#V\Delta_n=\infty$ and show that for a.e. trajectory of this process we have
\[
\lim_{n\to\infty}\frac{\eta_{\Delta_n}}{\#V\Delta_n}=m_{A}^\mathcal S \text{ in }\gcn.
\]
In particular, what happens when $\Delta_n$ is chosen  using the Bassino-Nicaud-Weil~\cite{BNW} process for generating random Stallings subgroup graphs? 
\end{prob}

\subsection{Dynamics of the $\Out(F_N)$-action on $\pgcn$}\label{subsec:dynamics}

For ordinary currents the dynamics of the $\Out(F_N)$-action on $\PCN$ and the interaction of this dynamics with that of the $\Out(F_N)$-action on $\cvnbar$ turned out to be particularly useful. For $\pgcn$ the dynamics of the $\Out(F_N)$-action appears to be more complicated than in the $\PCN$-case, particularly because, even without projectivization, $\Out(F_N)$ fixes the point $\eta_{F_N}\in \gcn$.

In~\cite{Martin} R.~Martin introduced the subset $\mathcal M_N\subseteq \PCN$ as the closure in $\PCN$ of the set
\[
\{[\eta_g]: g\in F_N \text{ is a primitive element} \}.
\]
It is easy to see that for every $N\ge 2$ $\mathcal M_N\subseteq \PCN$ is a closed $\Out(F_N)$-invariant nonempty subset.
In \cite{KL1} it was shown that for $N\ge 3$ $\mathcal M_N$ is the unique smallest such subset, that is, whenever $Z\subseteq \PCN$ is nonempty, closed and  $\Out(F_N)$-invariant, then $\mathcal M_N\subseteq Z$.

As noted in Remark~\ref{rem:curr} above, $\PCN\subseteq \pgcn$ is a closed $\Out(F_N)$-invariant subset and, similarly, $\Curr(F_N)\subseteq \gcn$ is a closed $\Out(F_N)$-invariant subset.  
The set $\{[\eta_{F_N}]\}$ is also a closed $\Out(F_N)$-invariant subset
of $\pgcn$, so that $\pgcn$ contains at least 2 minimal nonempty closed $\Out(F_N)$-invariant subsets.

\begin{prob}
Let $N\ge 3$. Characterize those $[\mu]\in \pgcn$ such that the closure of the $\Out(F_N)$-orbit in $\pgcn$ contains $\mathcal M_N$. Is it true that the closure of $\Out(F_N)[\mu]$ contains $\mathcal M_N$ if and only if $[\mu]\ne [\eta_{F_N}]$?
\end{prob}
As we note below, we do know that for every nontrivial finitely generated subgroup $H\le F_N$ of infinite index, the closure of $\Out(F_N)[\eta_H]$ does contain $\mathcal M_N$.

Let $N\ge 3$ and let $\phi\in \Out(F_N)$ be an atoroidal iwip (irreducible with irreducible powers) element. In this case it is known that $\phi$ has exactly two distinct fixed points in $\PCN$, the currents $[\mu_+]$ and $[\mu_-]$; moreover there are $\lambda_+, \lambda_- > 1$ such that $\phi\mu_+=\lambda_+\mu_+$ and $\phi^{-1}\mu_{-}=\lambda_-\mu_-$. 
R.~Martin proved~\cite{Martin} that the action of $\phi$ on $\PCN$ has \lq\lq North-South\rq\rq\ dynamics:
If $[\mu]\in \PCN, [\mu]\ne [\mu_-]$ then
$
\lim_{n\to\infty} \phi^n[\mu]=[\mu_+]
$
and
if $[\mu]\in \PCN, [\mu]\ne [\mu_+]$ then
$
\lim_{n\to\infty} \phi^{-n}[\mu]=[\mu_-].
$

As observed above, by Remark~\ref{rem:curr} we have $\mu_\pm\in \Curr(F_N)\subseteq \gcn$ and $[\mu_\pm]\in \PCN\subseteq \pgcn$.

The situation for $\pgcn$ is immediately complicated by the presence of the global fixed point $[\eta_{F_N}]$. For example, for any $c_1,c_2>0$ it is not hard to see that
\[
\lim_{n\to\infty} \phi^n[c_1\mu_-+c_2\eta_{F_N}]=[\eta_{F_N}]
\quad \text{ and } \quad
\lim_{n\to\infty} \phi^{-n}[c_1\mu_+ + c_2\eta_{F_N}]=[\eta_{F_N}].
\]

\begin{prob}
Let $N\ge 3$ and let $\phi\in\Out(F_N)$ be an atoroidal iwip.
Characterize those $[\mu]\in \pgcn$ for which
\[
\lim_{n\to\infty} \phi^n[\mu]=[\mu_+].
\]
Is it true that the above convergence holds for every $[\mu]\in \pgcn$ which is not of the form $[\mu]=[c_1\mu_-+c_2\eta_{F_N}]$ where $c_1,c_2\ge 0$ and $|c_1|+|c_2|>0$?
\end{prob}

Using the results of~\cite{BFH97} we can show that if $H\le F_N$ is a nontrivial finitely generated subgroup of infinite index, then
\[
\lim_{n\to\infty} \phi^n[\eta_H]=[\mu_+] \text{  and  } \lim_{n\to\infty} \phi^{-n}[\eta_H]=[\mu_-].
\]
Since $[\mu_+]\in \mathcal M_N$, this fact does imply that the closure in $\pgcn$ of the orbit $\Out(F_N)[\eta_H]$ contains $\mathcal M_N$.

\subsection{Subset algebraic laminations}\label{subsec:lamination}

An \emph{algebraic lamination} on $F_N$ is a closed $F_N$-invariant subset of the space of 2-elements subsets of $(\bd F_N)$.  Algebraic laminations proved to be useful objects in the study of $\Out(F_N)$ and of the Outer space. In particular, for every $T\in\cvnbar$ there is a naturally defined \emph{dual lamination} $L^2(T)$ which records some essential information about the geometry of the action of $F_N$ on $T$. 

For an arbitrary current $\mu\in\Curr(F_N)$ its \emph{support} is an algebraic lamination. It is proved in \cite{KL3} that for $\mu\in\Curr(F_N)$ and $T\in\cvnbar$ $\langle T,\mu\rangle=0$ if and only if $\supp(\mu)\subseteq L^2(T)$, and this fact plays a key role in understanding the interplay between the dynamics of $\Out(F_N)$-actions on $\CVNbar$ and $\PCN$.  

By analogy with the set-up discussed above, we say that a \emph{subset algebraic lamination} on $F_N$ is a closed $F_N$-invariant subset of $\mathfrak C_N$.
In particular, for any $\mu\in\gcn$, the support $\supp(\mu)$ is a subset algebraic lamination. Here $\supp(\mu):=\mathfrak C_N-U$ where $U$ is the largest open subset of $\mathfrak C_N$ such that $\mu(U)=0$.

\begin{prob}
Assuming that the answer to Problem~\ref{prob:covol} is positive,  given $T\in\cvnbar$ does there exist a naturally defined subset lamination $L^2_\mathcal S(T)\subseteq \mathfrak C_N$ which captures the information about all $\mu\in \gcn$ with $||\mu||_T=0$? Or at least some naturally defined subset lamination $L^2_\mathcal S(T)\subseteq \mathfrak C_N$ which captures the information about all finitely generated $H\le F_N$ with $||H||_T=0$?
If such a notion of $L^2_\mathcal S(T)$ does exist, how does $L^2_\mathcal S(T)$ look like for stable trees of various free group automorphisms?

Note that for any reasonable definition of $L^2_\mathcal S(T)$ one expects to have $L^2(T)\subseteq L^2_{\mathcal S(T)}$.
 \end{prob}

\subsection{Approximation and weight realizability problem}\label{subsec:weight}

We know by Theorem~\ref{thm:dense} that any $\mu\in \gcn$ can be approximated
by rational subset currents. It would be interesting to find
explicit procedures (e.g. algorithmic or probabilistic) for producing
such approximations in various specific contexts.

\begin{prob}
Let $A$ be a free basis of $F_N$ and let $m_A^\mathcal S$ be the corresponding absolute uniform current. Find a natural probabilistic process producing a sequence of core graphs $\Delta_n$ such that $\lim_{n\to\infty} [\mu_{\Delta_n}] =[m_A^\mathcal S]$ in $\pgcn$.
\end{prob}

In \cite{Martin} Martin proves that for a marking $\alpha:F_N\to \pi_1(\Gamma)$ and a nonzero geodesic current $\mu\in Curr(F_N)$ we have $( v; \mu)_\alpha\in \mathbb Z$ for every nontrivial reduced edge-path $v$ in $\Gamma$ if and only if $\mu=\eta_{g_1}+\dots +\eta_{g_m}$ for some $m\ge 1$ and some nontrivial $g_1,\dots, g_m\in F_N$.
It is natural to ask a similar question for subset currents:

\begin{prob}
Let  $\alpha:F_N\to \pi_1(\Gamma)$ be a marking and let $\mu\in \gcn$ be a nonzero subset current such that for every non-degenerate finite subtree $K$ of $\widetilde\Gamma$ we have $(K; \mu)_\alpha\in \mathbb Z$. Does this imply that $\mu=\eta_{H_1}+\dots +\eta_{H_m}$ for some $m\ge 1$ and some nontrivial finitely generated subgroups $H_1,\dots, H_m\le F_N$? 
\end{prob}

In the case of $\Curr(F_N)$, the proof uses Whitehead graphs for cyclic words. In the context of subset currents the corresponding objects turn out to be hyper-graphs, rather than graphs.

Recall that a \emph{hyper-graph} $\mathbf G$ is a pair $(\mathbf V,\mathbf E)$, where $\mathbf V$ is the set of vertices and $\mathbf E$ is the set of hyper-edges. Every hyper-edge $\mathbf e$ is a nonempty subset of $\mathbf V$, whose elements are said to be \emph{incident} to $\mathbf e$ in $\mathbf G$.
We can also think of $\mathbf e$ as the characteristic function of this subset, so that $\mathbf e: \mathbf V\to\{0,1\}$ with $\mathbf e(\mathbf v)=1$ if and only if $\mathbf v$ is incident to $\mathbf e$.

Now let $\alpha:F_N\isom \pi_1(\Gamma)$ be a marking, let $X=\widetilde \Gamma$ and let $K\in \mathcal K_\Gamma$ be a finite non-degenerate subtree of $X$.  Recall from Section \ref{sec:uniform} that $\iota(K)$ denotes the number of interior vertices in $K$. 

We say that an interior vertex $x$ of $K$ is a \emph{boundary-interior vertex} of $K$ if $x$ is an interior vertex of $K$,  $x$ is adjacent to a terminal edge of $K$ and removing from $K$ all terminal edges of $K$ incident to $x$ produces a tree in which $x$ has degree $1$.  That is, $x$ is an interior vertex of $K$ and the degree of $x$ in $K$ is equal to $p+1$ where $p$ is the number of terminal edges of $K$ incident to $x$.
Note that every finite $K$ with $\iota(K)\ge 2$ always has at least one boundary-interior vertex. 
For each boundary-interior vertex $x$ of $K$ let $U_x$ be the set of topological edges of $K$ that connect $x$ to vertices of degree 1 of $K$. Let $K_x$ be obtained from $K$ by removing all the edges of $U_x$. Thus $K=K_x\cup U_x$. Note that $x$ is a vertex of degree 1 in $K_x$, so that $\iota(K_x)=m-1$.

For each $m\ge 2$ define the \emph{level-$m$ initial hyper-graph}  $\mathbf G_{\alpha,m}$ as follows. The vertex set of $\mathbf G_{\alpha,m}(\mu)$ is the set $Z_{m-1}$ of all $F_N$-translation classes $[K]$ where $K\in\mathcal K_\Gamma$ is a tree with $\iota(K)=m$. The hyper-edge set is $Z_{m}$. Every $[K']\in Z_m$ defines the incidence function $\mathbf e_{[K']}: Z_{m-1}\to \{0,1\}$ as follows. For $[K]\in Z_{m-1}$ $\mathbf e_{[K']}([K])=1$ if and only if there is a boundary-interior vertex  $x$ of $K'$ such that $[K_x']=[K]$. 

For $m=1$ we can also define $\mathbf G_{\alpha,1}$ in a similar way. The vertex set is $Z_0$ and the hyper-edge set is $Z_1$, but the incidence is defined slightly differently. Namely, for $[K']\in Z_1$  and an element $[K]\in Z_0$ (note that $K$ is necessarily a single topological edge of $X$) we have $\mathbf e_{[K']}([K])=1$ if and only if there is a topological edge of $K'$ which is an $F_N$-translate of $K$.

For $m\ge 1$, a  \emph{weighted level-$m$ initial hyper-graph} consists of the hyper-graph $\mathbf G_{\alpha,m}$ endowed with the  \emph{weight functions} $\theta: Z_{m-1}\to \R_{\ge 0}$ and $\theta: Z_{m}\to \R_{\ge 0}$.

Every current $\mu\in \gcn$ defines a weighted hyper-graph $\mathbf G_{\alpha,m}(\mu)$, with the weight functions $\theta([K']):=( K'; \mu)_\alpha$,  $\theta([K]):=(K; \mu)_\alpha$, where $[K']\in Z_m$, $[K]\in Z_{m-1}$.


\begin{prob}\label{prob:wr}
Let $m\ge 1$ and let $(\mathbf G_{\alpha,m}, \theta)$ be a weighted initial level-$m$ hyper-graph such that the weight function $\theta$ is $\mathbb Z_{\ge 0}$--valued and satisfies the Kirchoff condition (see Proposition \ref{prop:kirch}) for every $[K]\in Z_{m-1}$.
Does there exist a finite $\Gamma$-core graph $\Delta$ such that $(\mathbf G_{\alpha,m}, \theta)=\mathbf G_{\alpha,m}(\mu_\Delta)$?
\end{prob}

As shown in \cite{Ka2}, in the $\Curr(F_N)$ context the analog of the above question has positive answer. In that case initial hyper-graphs are directed graphs and the Kirchoff condition implies the existence of an Euler circuit in every connected component of the graph, when the weights are treated as multiple edges. However, it is not clear what might serve as a substitute for Euler's theorem in the hyper-graph context, and even the proof of Theorem~\ref{thm:dense} does not appear to help with  Problem~\ref{prob:wr}.

\end{document}